              \def\version{6 April 2020}	        	%
\documentclass[reqno,11pt]{amsart} 
\usepackage[utf8]{inputenc}
\usepackage{amsmath} 
\usepackage[mathscr]{eucal}
\usepackage{amssymb}
\usepackage{srcltx} 
\usepackage{dsfont}
\usepackage{hyperref}
\usepackage{color}
\usepackage{enumerate}
\usepackage{tikz}
\usepackage{comment}
\usepackage{lscape}
\usepackage{graphicx}

\numberwithin{equation}{section}
 
\def\emptyset{\varnothing} 
\def\d{{\rm d}} 
\def\e{\varepsilon} 
 
\newfam\Bbbfam 
\font\tenBbb=msbm10 
\font\sevenBbb=msbm7 
\font\fiveBbb=msbm5 
\textfont\Bbbfam=\tenBbb 
\scriptfont\Bbbfam=\sevenBbb 
\scriptscriptfont\Bbbfam=\fiveBbb

\def\2{\mathbf 2}

\newcommand{\R}     {\mathbb{R}} 
 
\newcommand{\N}     {\mathbb{N}} 
\renewcommand{\P}   {\mathbb{P}} 
 
\newcommand{\E}     {\mathbb{E}}

\newcommand{\smfrac}[2]{\textstyle{\frac {#1}{#2}}}

\def\1{{\mathchoice {1\mskip-4mu\mathrm l}      % Blackboard bold 1 
{1\mskip-4mu\mathrm l} 
{1\mskip-4.5mu\mathrm l} {1\mskip-5mu\mathrm l}}} 
 
\def\comment#1{} 
\newtheoremstyle{thm}{2ex}{2ex}{\itshape\rmfamily}{} 
{\bfseries\rmfamily}{}{1.7ex}{} 
 
\newtheoremstyle{rem}{1.3ex}{1.3ex}{\rmfamily}{} 
{\itshape\rmfamily}{}{1.5ex}{}

\renewcommand{\theequation}{\thesection.\arabic{equation}} 
 
\newtheorem{theorem}{Theorem}[section] 
\newtheorem{lemma}[theorem]{Lemma} 
\newtheorem{prop}[theorem] {Proposition}

\theoremstyle{definition}

\renewcommand{\d}{{\rm d}} 
 
\newcommand{\eps}{\varepsilon}

\newcommand{\Tr}{{\operatorname {Tr}\,}}

\newcommand{\Bcal}  {{\mathcal B}}
 
\newcommand{\Dcal}   {{\mathcal D }}

\newcommand{\Lcal}   {{\mathcal L }} 
 
\newcommand{\Ncal}   {{\mathcal N }}

\newcommand\numberthis{\addtocounter{equation}{1}\tag{\theequation}}
\renewcommand{\e}   {{\operatorname e }}

\definecolor{Red}{rgb}{1,0,0}

\begin{document} 
 
\title[Invasion and fixation of microbial dormancy traits under competitive pressure]{Invasion and fixation of microbial dormancy traits \\ under competitive pressure}
\author[Jochen Blath and András Tóbiás]{}
\maketitle
\thispagestyle{empty}
\vspace{-0.5cm}

\centerline{\sc Jochen Blath and András Tóbiás{\footnote{TU Berlin, Straße des 17. Juni 136, 10623 Berlin, {\tt blath@math.tu-berlin.de, tobias@math.tu-berlin.de}}}}
\renewcommand{\thefootnote}{}
\vspace{0.5cm}
\centerline{\textit{TU Berlin}}

\bigskip

\centerline{\small(\version)} 
\vspace{.5cm} 
 
\begin{quote} 
{\small {\bf Abstract:}} Microbial dormancy is an evolutionary trait that has emerged independently at various positions across the tree of life. It describes the ability of a microorganism to switch to a metabolically inactive state that can withstand unfavorable conditions. However, maintaining such a trait requires additional resources that could otherwise be used to increase e.g.\ reproductive rates. 
In this paper, we aim for gaining a basic understanding under which conditions maintaining a seed bank of dormant individuals provides a ``fitness advantage'' when facing resource limitations and competition for resources among individuals (in an otherwise stable environment). In particular, we wish to understand when an individual with a  ``dormancy trait'' can invade a resident population lacking this trait despite
having a lower reproduction rate than the residents.
%the fact that maintaining a dormancy-related machinery reduces its reproduction rate in comparison to the residents. 
To this end, we follow a stochastic individual-based approach employing birth-and-death processes, where dormancy is triggered by competitive pressure for resources. In the large-population limit, we identify a necessary and sufficient condition under which a complete invasion of mutants has a positive probability. Further, we explicitly determine the limiting probability of invasion and the asymptotic time to fixation of mutants in the case of a successful invasion. In the proofs, we observe the three classical phases of invasion dynamics in the guise of Coron et\ al.\ (2017, 2019).
\end{quote}

%\vfill



\setcounter{tocdepth}{3}

%\tableofcontents

\setcounter{section}{0}
\begin{comment}{
This is not visible.}
\end{comment}
\section{Introduction}
\label{sec-introduction} %\color{red} Temporary introduction, feel free to ignore it completely \color{black}

Dormancy is an evolutionary trait that has emerged independently at various positions across the tree of life. In the present article, we are in particular interested in microbial dormancy (cf.\ \cite{LJ11} and \cite{SL18} for recent overviews of this subject). Microbial dormancy describes the ability of a microorganism to switch to a metabolically inactive state in order to withstand unfavourable conditions (such as resource scarcity and competitive pressure or extreme environmental fluctuations), and this seems to be a highly effective (yet costly) evolutionary strategy. In certain cases, for example in marine sediments, simulation studies indicate that under oligotrophic conditions, the fitness of an organisms is determined to a large degree by its ability to simply stay alive, rather than to grow and reproduce (cf.\ \cite{BAL18}). Indeed, maintaining a dormancy trait requires additional resources in comparison to individuals lacking this trait, resulting in significant trade-offs such as e.g.\ a lower reproduction rate. %Yet, recent simulation results in~\cite{SFCT16} suggest that the ability of effectively switching between activity and dormancy adapted to competition and environmental changes may compensate for such disadvantages. \color{red} I don't mind getting rid of \cite{SFCT16}. \cite{BAL18} seems important. \color{black}

In this paper, we aim at gaining a basic rigorous understanding for the conditions under which maintaining a dormancy trait can be beneficial. We investigate the particular question whether an individual with a dormancy trait can invade a resident population lacking this trait, even if maintaining dormancy reduces its reproduction rate compared to the rate of the residents, under otherwise stable environmental conditions.  
To this end, we follow a stochastic individual-based approach employing birth-and-death processes (a classic set-up underlying much of adaptive dynamics, as outlined e.g.\ in \cite{B19}), where dormancy is triggered in response to competitive pressure for limited resources. In the large-population limit, we identify a necessary and sufficient condition under which the invasion of mutants, despite having a lower reproduction rate than the resident population, has a positive probability. Further, we explicitly determine the limiting probability of invasion and the asymptotic time of fixation of mutants in the case of a successful invasion.

To be more explicit, in our model the total population evolves according to a continuous time Markov chain. Initially, there is a fit resident population, which we assume to be close to its equilibrium population size, featuring (random) reproduction, natural death (``death by age''), and death by competition. This results in a stochastically evolving population with logistically regulated drift fluctuating around a constant carrying capacity (reflecting a stable yet limited supply of resources). We assume that environmental conditions are also stable and do not affect reproduction, death or competition rates. In this situation, we then assume that a single ``mutant'' (or ``migrant'') with ``dormancy trait'' appears in the population, who on the one hand is still fit enough to survive in absence of the residents (however with a strictly lower reproduction rate), but on the other hand is able to switch to a dormant state at a rate proportional to the ``competitive pressure'' exerted on her due to crowding and limited resource availability. That is, for some $0<p<1$, ``competition events'' that would normally cause death for an ordinary resident individual kill a mutant individual only with probability $1-p$. Otherwise, with probability $p$, the mutant individual affected by competition will persist and switch to the dormant state. Finally, dormant mutant individuals neither reproduce nor are affected by competitive pressure for resources while they are still to some degree exposed to natural death (at a rate typically smaller than for active individuals). We assume that at a constant ``resuscitation rate''%\color{blue} Das könnte man natuerlich auch noch von der competitive pressure abhaengig machen, waere evtl interessant!\color{black}
, they switch back to the active state. %For simplicity, we ignore further mutations both in the active and in the mutant population, similarly to other mathematical works on speciation \cite{C+16} and the emergence of homogamy \cite{C+19} in individual-based stochastics models. 

Our main results show that the mutants will invade the resident population with positive probability under a suitable condition on the parameters of the model. This condition has the following interpretation: the advantage of the resident population caused by its higher reproduction rate needs to be over-compensated by the advantage of the mutant population resulting from being able to escape competitive deaths due to overcrowding by switching into dormancy. This condition can be made entirely transparent in terms of the parameters of the model, see \eqref{invasionpossible} resp.\ Section~\ref{sec-discussion} below. Under this condition, we characterize the probability of invasion (that is, the mutants completely replace the residents and reach their own equilibrium carrying capacity), and we identify the expected time of invasion on a logarithmic scale in the large-population limit. With high probability, a successful mutation follows the three classical phases exhibited in basic adaptive dynamics models (which were introduced in \cite[Section 3]{C06}; see e.g.\ \cite[Section 4.1]{B19} for a slightly more general picture, but in particular \cite{C+16, C+19} for work in a closely related context that inspired our analysis and provides many of the necessary tools): (1) mutant growth until reaching a population size comparable to the carrying capacity, while during the same time period the resident population stays close to its equilibrium size, (2) a phase where all sub-populations are large and the dynamics of the frequency process can be approximated by a deterministic dynamical system, (3) extinction of the resident population, while the mutant population remains close to its equilibrium size.

Note that for our results it is essential that switching into dormancy is induced by competitive pressure.  Indeed, if instead this switching happens at a constant rate (``spontaneous'' or ``stochastic switching'', cf.\ e.g.\ \cite{LJ11}), the mutants will {\em never} be able to invade the resident population unless their birth rate is higher than that of the residents (in which case their invasion would also be possible without a dormancy trait, and the assumption that dormancy is a costly trait would be violated). Further, mutants cannot make the residents go extinct unless they are fit enough to survive on their own; thus, evolutionary suicide, as observed e.g.\ in \cite{BCFMT16}, does not occur in our model. Long-term coexistence of residents and mutants is also excluded in our modelling set-up.

%In this case, the mutant population will die out with positive probability, but also with positive probability, it will invade the population and make the resident population die out. The invasion consists of three phases: first, the mutant population dies out or grows to an order that is comparable to the carrying capacity, which takes an order of time that is logarithmic in the carrying capacity. Second, given that it has survived the first phase, in a bounded order of time the mutant population grows to close to its equilibrium population size and the resident population becomes small compared to the carrying capacity. Third, the resident population dies out, which again requires a logarithmic order of time. In the limit of a large population size (equivalently, large carrying capacity), we identify the precise prefactors of the logarithm for the lengths of these times, and we give an explicit characterization for the invasion probability.

Let us note that while dormancy was recently investigated in several mathematical works in the area of population genetics and coalescent theory (see e.g.~\cite{KKL01, BGKS13, BGKW16, KATZ17, BGKW18}), in the field of adaptive dynamics we are not aware of prior work involving dormancy. The present paper takes a first step in this direction, analysing the invasion dynamics in a simple toy model. In order to make this model more realistic, one could e.g.~incorporate further mutations in the spirit of adaptive dynamics. In the regime of very rare mutations introduced by Champagnat (cf.~\cite{C06, CM11, BBC17}), we expect that the model behaves similarly to the case of no further mutation. Recently, in \cite{CMT19}, a regime of still rare but more frequent mutations was considered, with the additional effect of horizontal gene transfer. Here, mutation rates are large enough so that  small sub-populations can have macroscopic effects on the whole population. It should be interesting to study the additional effects of dormancy traits in this regime. As a further step, one could also introduce spatiality in the model, which is relevant in modelling the trait space (see e.g.~\cite{BB18}) or the environment of the populations (see e.g.~\cite{FM04,CM07}). Finally, the resuscitation rate, which is assumed constant in the present paper, could also be made dependent on the strength of competition.

Note that related scenarios involving ``phenotypic switches'', arising e.g.\ in cancer modelling, have been analysed recently by \cite{BB18,HBT13}. For dormancy and switching models in {\em fluctating environments}, dynamical systems and branching process models have been investigated in \cite{MS08,DMB11}. Here, as in the {\em competition} setup of the present paper, the basis of a rigorous understanding for the evolutionary advantages of seed banks seems to be emerging.
%\color{blue} Comment on dormancy for plants a la Tellier?\color{black}
It seems fair to say that dormancy in its many forms, and its interplay with other evolutionary and ecological forces, will provide many interesting future research challenges in mathematical biology.

The remainder of this paper is organized as follows. In Section~\ref{sec-deathorseedbank} we introduce our model and state our main results. Next, in Section~\ref{sec-discussion} we discuss some strongly related questions. Finally, in Section~\ref{sec-proof} we prove the main results. Each of these sections starts with a description about its internal organization.

\section{Model definition and main results}\label{sec-deathorseedbank}
The structure of this section is the following. In Section~\ref{sec-modeldef} we define our stochastic population model. Next, the goal of Section~\ref{sec-invasionconditions} is to introduce necessary and sufficient conditions for mutant invasion with positive probability, to present the formulas for the probability and time of invasion in the large-population limit, and to provide a heuristic justification for these. In particular, we comment on the probability and time of the invasion. The introduced quantities and conditions are then used in Section~\ref{sec-statementofresult} in order to state our main results, Theorems~\ref{thm-invasion}, \ref{thm-success}, and \ref{thm-failure}, the proof of which will make our heuristic arguments rigorous.

\subsection{The model}\label{sec-modeldef}

We have two traits, the resident one (1) and the mutant one (2). Mutant individuals can have an active (2a) and a dormant (2d) state. As an interpretation, we will sometimes say that the dormant individuals are in the \emph{seed bank}. Informally speaking, the model is defined as follows.
\begin{itemize}
    \item A resident individual gives birth to another such individual at rate $\lambda_1>0$.
    \item An active mutant individual gives birth to another such individual at rate $\lambda_2 \in (0,\lambda_1)$.
    \item Any active individual has a natural death rate $\mu \in (0,\lambda_2)$.
    \item $K>0$ is the carrying capacity of the population.
    \item The competitive pressure felt by an active individual from another active individual is $\alpha/K>0$, where $\alpha>0$. For any ordered pair $(x_i,x_j)$ of active individuals, at rate $\alpha/K>0$ a competitive event affecting $x_i$ happens. We fix $p \in (0,1)$. At a competitive event, in case $x_i$ is a resident individual, it dies. If $x_i$ is a mutant individual, it dies with probability $1-p$ and becomes a dormant (mutant) individual with probability $p$. \\
    In other words, writing $\N_0=\{ 0,1,2,\ldots \}$ and $\N=\{ 1,2,\ldots\}$, in a population with $n_{1} \in \N_0$ (active) resident individuals and $n_{2a} \in \N_0$ active mutant individuals, writing $n_a=n_1+n_{2a}$ for the total number of active individuals, a resident individual dies by competition at rate $\alpha n_a/K$, an active mutant dies by competition at rate $(1-p)\alpha n_a/K$ and switches to dormant mutant at rate $p\alpha n_a/K$. 
    \item For some $\kappa \geq 0$, a dormant (mutant) individual dies at rate $\kappa \mu$.
    \item A dormant (mutant) individual becomes an active (mutant) individual at rate $\sigma>0$. 
\end{itemize}
Further necessary conditions on the parameters will be specified later in the sequel.

To be more precise, we consider, for $t \geq 0$, a finite number $N_t \in \N_0$ of individuals $\{ x_i \colon i \in [N_t]\}$, where for all $i \in [N_t]$ we have $x_i \in \{ 1, 2a, 2d \}$. Here we wrote $[n]=\{ 1,2,\ldots,n\}$ for $n \in \N_0$, in particular, $[0]=\emptyset$. We define the triple of rescaled frequency processes
\[ (\mathbf N_t^K)_{t \geq 0} = ((N_{1,t}^K,N_{2a,t}^K,N_{2d,t}^K))_{t \geq 0}, \]
where for $x \in \{ 1, 2a, 2d \}$,
\[ N_{x,t}^K = \frac{1}{K} \# \{ x_i \colon i \in [N_t], x_i= x\} \]
is the number of individuals of type $x$ rescaled by $K$. We also write
\[ N_{2,t}^K = N_{2a,t}^K + N_{2d,t}^K \]
for $1/K$ times the total population size of mutant individuals and
\[ N_t^K = N_{1,t}^K + N_{2,t}^K = \frac{N_t}{K} \]
for $1/K$ times the total population size.
Hence, $\mathbf N_t^K$ is a $\big(\frac{1}{K} \N\big)^3$-valued Markov process with transitions
\begin{align*}
    (n_1,n_{2a},n_{2d}) \to
    \begin{cases}
    & (n_1+\smfrac{1}{K},n_{2a},n_{2d}) \text{ at rate } K n_1\lambda_1, \\
    & (n_1,n_{2a}+\smfrac{1}{K},n_{2d}) \text{ at rate } K n_{2a}\lambda_2, \\
    & (n_1-\smfrac{1}{K},n_{2a},n_{2d}) \text{ at rate } K n_1(\mu + \alpha (n_1+n_{2a})), \\
    & (n_1,n_{2a}-\smfrac{1}{K},n_{2d}) \text{ at rate } K n_{2a}(\mu + (1-p)\alpha (n_1+n_{2a})), \\
    & (n_1,n_{2a}-\smfrac{1}{K},n_{2d}+\smfrac{1}{K}) \text{ at rate } K n_{2a} p\alpha (n_1+n_{2a}), \\
    & (n_1,n_{2a},n_{2d}-\smfrac{1}{K}) \text{ at rate } K n_{2d}\kappa\mu, \\
    & (n_1,n_{2a}+\smfrac{1}{K},n_{2d}-\smfrac{1}{K}) \text{ at rate } K n_{2d}\sigma.
    \end{cases}
\end{align*}
%with infinitesimal generator defined for all bounded measurable functions $f \colon \big(\frac{1}{K} \N\big)^3 \to \R$ as
%\begin{align*}
%& L^K f(n_1,n_{2a},n_{2d})  = \Big( f\big( n_1+\smfrac{1}{K}, n_{2a}, n_{2d} \big) - f(n_1,n_{2a},n_{2d}) \Big) \lambda_1 \\ &  + \Big( f\big( n_1, n_{2a}+\smfrac{1}{K}, n_{2d} \big) - f(n_1,n_{2a},n_{2d}) \Big) \lambda_1  \\ &  + \Big( f\big( n_1-\smfrac{1}{K}, n_{2a}, n_{2d} \big) - f(n_1,n_{2a},n_{2d}) \Big) \mu  +\Big( f\big( n_1, n_{2a}-\smfrac{1}{K}, n_{2d} \big) - f(n_1,n_{2a},n_{2d}) \Big) \mu \\
%&  + \Big( f\big( n_1-\smfrac{1}{K}, n_{2a}, n_{2d} \big) - f(n_1,n_{2a},n_{2d}) \Big) \alpha (n_1+n_{2a}) \\ &  + \Big( f\big( n_1, n_{2a}-\smfrac{1}{K}, n_{2d} \big) - f(n_1,n_{2a},n_{2d}) \Big) (1-p) \alpha (n_1+n_{2a}) 
%\\ &  + \Big( f\big( n_1, n_{2a}-\smfrac{1}{K}, n_{2d}+\smfrac{1}{K} \big) - f(n_1,n_{2a},n_{2d}) \Big) p \alpha (n_1+n_{2a}) 
%\\ &  + \Big( f\big( n_1, n_{2a}, n_{2d}-\smfrac{1}{K} \big) - f(n_1,n_{2a},n_{2d}) \Big) \kappa \mu 
 %+ \Big( f\big( n_1, n_{2a}+\smfrac{1}{K}, n_{2d}-\smfrac{1}{K} \big) - f(n_1,n_{2a},n_{2d}) \Big) \sigma.
% \numberthis\label{generatorcomp}
%\end{align*}
\subsection{Assumptions and heuristics}\label{sec-invasionconditions}
The Markov process $(\mathbf N_t^K)_{t \geq 0}$ is well-defined for any $K>0$, given the initial condition. Relevant initial conditions satisfy $\bar N_0^K \approx (\bar n_1, \frac{1}{K}, 0)$ where $\bar n_1$ is the equilibrium population size of the resident population in absence of the mutant population. That is, at time $0$, resident individuals are close to equilibrium, and there is precisely one active mutant and there are no dormant mutants.

Now, we want to find necessary and sufficient conditions under which the probability of mutant invasion is nonvanishing in the large-population limit. Further, conditional on a successful invasion, we want to identify the time of invasion for large $K$ on the logarithmic scale. To this aim, we have to choose the parameters in such a way that, roughly speaking, the following assertions hold.
\begin{enumerate}
    \item\label{residentslive} The resident population is able to survive on its own, i.e., $\bar n_1>0$.
    \item Mutants are also fit: their equilibrium population size $(\bar n_{2a},\bar n_{2d})$ is coordinatewise positive.
    \item\label{firstphase-short} \emph{Phase I of the invasion}: For large $K$, starting from $\bar N_0^K \approx (\bar n_1, \frac{1}{K}, 0)$, the probability that $N_{2,t}^K=0$ eventually is not close to one for large $K$.
    \item\label{secondphase-short} \emph{Phase II}: Given that the total mutant population has reached size $\eps K$, for $\eps>0$ small, with high probability $\bar N_t^K$ will get close to $(K \eps,K \bar n_{2a},K \bar n_{2d})$ for arbitrarily small $\eps>0$.
    \item\label{thirdphase-short} \emph{Phase III}: Given that the process reached the state $(K \eps,K \bar n_{2a},K \bar n_{2d})$, the resident population will die out with high probability.
\end{enumerate}
Let us now heuristically identify the conditions corresponding to \eqref{residentslive}--\eqref{thirdphase-short}. The conditions that are necessary and sufficient for \eqref{firstphase-short} will turn out also to be sufficient for \eqref{secondphase-short} and \eqref{thirdphase-short}. These heuristics will be made precise during the proof of the main results of the paper.

%. We want to choose the parameters $\lambda_1,\mu$, and $\alpha$ in such a way that there is a unique positive equilibrium $\bar n_1$ of the resident population and this one is asymptotically stable. Further, we want to choose $\lambda_1,\mu,p,\kappa$, and $\sigma$ in such a way that the mutant population size (which is the sum of the number of active and dormant individuals) has a unique positive equilibrium $n_2^*$ also being asymptotically stable. Finally, we want that close to time zero the mutant population can be approximated by a(n at least slightly) supercritical branching process, so that there is positive probability for an 
%invasion of the population by the mutants and subsequent extinction of the residents (but not the mutants). Then we want to determine/bound these invasion probabilities.
%Let us now identify necessary and sufficient conditions for these properties. In the rest of this section, we heuristically assume that all the mentioned couplings suitably approximate our Markov process, and we introduce notations that are necessary in order to state our results. Why these approximations are suitable will be explained in the proof of our main result, Theorem~\ref{thm-invasion}.
\begin{enumerate}
    \item In absence of mutants, for large $K$, the rescaled resident population $N_{1,t}^K$ can be approximated by $n_1(t)$, where $n_1(\cdot)$ solves the quadratic ODE
    \[ \dot{n}_1(t)=n_1(t)(\lambda_1-\mu- \alpha n_1(t)). \]
    If $\lambda_1>\mu$, this system has a unique positive equilibrium, given as
    \[ \bar n_1 = \frac{\lambda_1-\mu}{\alpha}, \]
    which is also asymptotically stable. Else, there is no stable positive equilibrium.
    \item\label{secondpoint} Similarly, in absence of residents, for large $K$, the rescaled mutant population $(N_{2a,t}^K,N_{2d,t}^K)$ can be approximated by 
    $(n_{2a}(t),n_{2d}(t))$, where  $(n_{2a}(\cdot),n_{2d}(\cdot))$ solves the two-dimensional system of ODEs
    \[
    \begin{aligned}
    \dot n_{2a}(t) &  = n_{2a}(t)(\lambda_2-\mu-\alpha n_{2a}(t))+\sigma n_{2d}(t),\\
    \dot n_{2d}(t) & = p\alpha n_{2a}(t)^2 - (\kappa \mu+\sigma) n_{2d}(t).
    \end{aligned}
    \numberthis\label{linearized}
    \]
    Linearizing this system, we obtain the Jacobian matrix
    \[ A(n_{2a},n_{2d})= \left(  \begin{smallmatrix}\lambda_2-\mu-2\alpha n_{2a} & \sigma \\ 2p\alpha n_{2a} & -\kappa \mu-\sigma\end{smallmatrix} \right). \numberthis\label{Jacobian} \]
    Clearly, there is no equilibrium of the form $(0,\cdot)$ or $(\cdot,0)$ apart from $(0,0)$. Further, we have
    \[ A(0,0) = \left(  \begin{smallmatrix}\lambda_2-\mu & \sigma \\ 0 & -\kappa \mu-\sigma\end{smallmatrix} \right). \]
    %Thus, $(0,0)$ is asymptotically stable whenever $\det A(0,0)=(\lambda_2-\mu)(-\kappa \mu-\sigma)>0$, i.e., $\lambda_2<\mu$ (in this case it is easy to check that there are two negative eigenvalues), and stable whenever $\det A(0,0) \geq 0$, i.e., $\lambda_2 \leq \mu$. This holds for any $\sigma>0$ and $\kappa \geq 0$. 
    For $\lambda_2>\mu$, it is easy show that $A(0,0)$ has one negative and one positive eigenvalue and hence $(0,0)$ is unstable. Let us now show that for $\lambda_2>\mu$ we have a unique (coordinatewise) positive equilibrium, which is asymptotically stable.
    %given as \[ \bar n_2 = (\bar n_{2a}, \bar n_{2d})= \] 
    For an equilibrium $(n_{2a},n_{2d})$ with $n_{2a} \neq 0$, dividing both equations in \eqref{linearized} by $n_{2a}$, we obtain 
    \[ \frac{n_{2d}}{n_{2a}}=-\frac{\lambda_2-\mu-\alpha n_{2a}}{\sigma}=\frac{p\alpha n_{2a}}{\kappa \mu+\sigma}.  \numberthis\label{positiveequilibrium} \]
    From \eqref{positiveequilibrium} we obtain that there is precisely one such equilibrium, with coordinates
    %any equilibrium $(\bar n_{2a},\bar n_{2d})$ with $\bar n_{2a} \neq 0$ satisfies
    %\[ \bar n_{2a} = \frac{\mu-\lambda_2}{\sigma} \Big( \frac{p\alpha}{\kappa\mu+\sigma}-\frac{\alpha}{\sigma} \Big)^{-1}. \numberthis\label{xaformula} \]
    %Under the assumption that $\lambda_2>\mu$, this is positive. Indeed, since $p \in (0,1)$ and $\kappa \geq 0$, 
    %\[ \frac{p\alpha}{\kappa\mu+\sigma}-\frac{\alpha}{\sigma} \leq \frac{\alpha(1-p)}{\sigma}<0. \]
    %Now, $\bar n_{2a}$ can be expressed from \eqref{xaformula} as
    \[ \bar n_{2a} = \frac{(\lambda_2-\mu)(\kappa\mu+\sigma)}{\alpha(\kappa\mu+(1-p)\sigma)}>0, \qquad \bar n_{2d} = \frac{(\lambda_2-\mu)^2 p (\kappa\mu+\sigma)}{\alpha (\kappa\mu+(1-p)\sigma)^2}>0.\]
    were we used that $\lambda_2>\mu$, $\kappa\mu \geq 0$, $\sigma>0$ and $p \in (0,1)$.
    Comparing this to \eqref{Jacobian}, we obtain
    \[ \det A(\bar n_{2a}, \bar n_{2d}) = (\kappa \mu+\sigma)(\lambda_2-\mu). \]
    If $\lambda_2>\mu$, then the right-hand side is positive. In this case there are two strictly negative eigenvalues. This is true because the trace $\Tr A(\bar n_{2a}, \bar n_{2d})$ is negative, which follows from the fact that $\bar n_{2a}>\lambda_2-\mu$ and $\kappa\mu+\sigma>0$. Hence, $(\bar n_{2a},\bar n_{2d})$ is asymptotically stable.
    
    \item\label{thirdpoint} As long as the mutant population size $K N_{2,t}^K$ is negligible compared to $K$, the resident population can be approximated by its equilibrium population size, and the competition pressure felt by a mutant individual comes essentially  only from the resident population. This implies that the dynamics of the mutant population size process $(KN_{2a,t}^K,KN_{2d,t}^K)$ can be approximated by a bi-type linear branching process $(\widehat Z_{2a}(t),\widehat Z_{2d}(t))$ with rates
    \[ (n_{2a},n_{2d}) \to
    \begin{cases}
    & (n_{2a}+1,n_{2d}) \text{ at rate } n_{2a} \lambda_2, \\
    & (n_{2a}-1,n_{2d}) \text{ at rate } n_{2a} (\mu+\alpha \bar n_1(1-p)), \\
    & (n_{2a}-1,n_{2d}+1) \text{ at rate } n_{2a} \bar n_1\alpha p,\\
    & (n_{2a}+1,n_{2d}-1) \text{ at rate } \sigma n_{2d}, \\
    & (n_{2a},n_{2d}-1) \text{ at rate } \kappa \mu n_{2d}. 
    \end{cases}
    \] By classical results on multitype branching processes  \cite[Section 7.2]{AN72}, the process is supercritical, i.e., there is no almost sure convergence to $(0,0)$, if and only if the following \emph{mean matrix} has a positive eigenvalue
    \[ J= \begin{pmatrix}
    \lambda_2-\mu-\alpha \bar n_1 &  p\alpha \bar n_1 \\
    \sigma & -\kappa \mu-\sigma \\
    \end{pmatrix}
    = 
    \begin{pmatrix}
    \lambda_2-\lambda_1 & p(\lambda_1-\mu) \\
   \sigma   & -\kappa \mu-\sigma \\
    \end{pmatrix}. \numberthis\label{Jdef}
    \]
    In the interesting case $\lambda_2<\lambda_1$, it is impossible that we have two positive eigenvalues because $\Tr J<0$ follows from the definition of $\bar n_1$. To describe the condition that $J$ has a positive eigenvalue more explicitly, let us first consider the sign of $\det J$. In case there is precisely one positive eigenvalue, the determinant must be negative, which is equivalent to
    \[ \lambda_1-\lambda_2 < p(\lambda_1-\mu)\frac{\sigma}{\kappa\mu+\sigma}= p \alpha \bar n_1 \frac{\sigma}{\kappa\mu+\sigma}. \numberthis\label{invasionpossible}\]
    Indeed, since $\kappa\geq 0$, the eigenvalue equation in the variable $\lambda$ corresponding to the matrix $J$ in \eqref{invasionpossible} is
    \[ \lambda^2 + (\lambda_1-\lambda_2+\kappa\mu+\sigma) \lambda + \det J = 0. \]
    This quadratic equation always has two different real solutions if $\det J$ is negative, and hence one of the eigenvalues of $J$ must indeed be positive if \eqref{invasionpossible} holds. 
    The condition~\eqref{invasionpossible} turns out to be necessary and sufficient for the invasion probability to be asymptotically positive. We will interpret it and discuss the related notion of invasion fitness in Section~\ref{sec-invasionpossible}.
    \item Now we argue that under condition \eqref{invasionpossible}, given that the total mutant population has reached a population size of order $K$, the second phase of invasion also takes place, which ends with $\mathbf N_t^K \approx (0,\bar n_{2a},\bar n_{2d})$. In that phase, as long as all sub-populations are of order $K$, the process $\mathbf N_t^K$ can be approximated, for $K$ large, by the (deterministic) Lotka--Volterra type system 
    \[
    \begin{aligned}
    \dot n_{1}(t)&= n_1(t)(\lambda_1-\mu-\alpha (n_1(t)+n_{2a}(t)), \\
    \dot n_{2a}(t)& = n_{2a}(t)(\lambda_2-\mu-\alpha (n_1(t) + n_{2a}(t))+\sigma n_{2d}(t), \\
    \dot n_{2d}(t)&=p\alpha n_{2a}(t)(n_1(t)+n_{2a}(t))-(\kappa\mu+\sigma) n_{2d}(t).
    \end{aligned} \numberthis\label{3dimlotkavolterra}
    \]
    We will show below (see Proposition~\ref{prop-stableequilibrium3D}) that \eqref{invasionpossible} with $\lambda_1>\lambda_2>\mu$ is also sufficient to guarantee that this system has only one stable nonnegative equilibrium, which is 
        equal to $(0,\bar n_{2a},\bar n_{2d})$ and asymptotically stable. Moreover, there is a set of initial conditions that $\mathbf N_t^K$ reaches with high probability given that the mutants survived the first phase, such that starting from this set, the solution of \eqref{3dimlotkavolterra} tends to $(0,\bar n_{2a},\bar n_{2d})$ as $t \to \infty$.
    \item\label{lastpoint} After the second phase of invasion, the population rescaled by $1/K$ is close to the equilibrium $(0,\bar n_{2a},\bar n_{2d})$. To be more precise, the resident population size is of order $\eps K$ for some $\eps>0$ small.
     It remains to show that for large $K$, with probability tending to one, the resident population dies out within $O(\log K)$ time, while the mutant population stays close to equilibrium. Now, as long as $(KN_{2a,t}^K,KN_{2d,t}^K)$ is near $(K \bar n_{2a}, K \bar n_{2d})$ and the resident population is small compared to $K$, the competitive pressure that the resident individuals feel comes essentially only from the mutant population. This implies that $K N_{1,t}^K$ can be approximated by a branching process $\widehat Z_1(t)$ with rates
    \[ n \to \begin{cases} n+1 &\text{ at rate } n_1 \lambda_1, \\ n-1 & \text{ at rate } n_1(\mu+\alpha \bar n_{2a}) \end{cases}. \]
    In order to show that this branching process goes extinct almost surely, we have to verify that it is subcritical, i.e., the rate $n \to n+1$ is smaller than the rate $n \to n-1$. But this assertion is equivalent to the inequality \eqref{invasionpossible}. 
    %\[ \lambda_1- \mu - \alpha \bar n_{2a} =\frac{(\lambda_1-\mu)(\kappa\mu+(1-p)\sigma)-(\lambda_2-\mu)(\kappa\mu+\sigma)}{(\kappa \mu+(1-p)\sigma)}, \numberthis\label{subcriticality} \]
    %in order to see that \eqref{subcriticality} is negative, it suffices to show that
   % \[ (\lambda_1-\mu)(\kappa\mu+(1-p)\sigma)<(\lambda_2-\mu)(\kappa\mu+\sigma). \]
   % But this inequality is equivalent to \eqref{invasionpossible}. That is, under that condition the active mutant population size is always larger than the mutant one.
   % The extinction time of the branching process is asymptotically close to $\log K$ times the opposite of the expression in \eqref{subcriticality}. \\
   \item Using our multitype branching process approach, now we can compute the extinction probabilities under condition~\eqref{invasionpossible} with $\lambda_1>\lambda_2>\mu$. Define
\[ q = \P\big( \exists t < \infty \colon \widehat Z_{2a}(t)+\widehat Z_{2d}(t)=0 \big|  (\widehat Z_{2a}(0),\widehat Z_{2d}(0))=(1,0) \big). \numberthis\label{qdef} \]
By \cite[Section 7]{AN72}, $q$ is the first coordinate of the unique solution of the system of equations
\[ \numberthis\label{extinctionequation}
\begin{aligned}
 \lambda_2 (s_a^2-s_a) + p(\lambda_1-\mu) (s_d-s_a) + (\mu+(1-p)(\lambda_1-\mu))(1-s_a) &=0, \\
\sigma (s_a-s_d) + \kappa \mu(1-s_d) &=0,  \\
\end{aligned}
\]
in $[0,1]^2 \setminus \{ (1,1) \}$,
while the second coordinate of the same solution is the extinction probability given that the branching process is started from $(0,1)$. We will discuss the dynamics of the process started with one dormant individual and interpret the relation between $s_a$ and $s_d$ in Section~\ref{sec-discussion-startingwithadormant}.

%Analogously, if spontaneous switching at rate $\sigma$ is also present in the model, then the extinction probability $q'$ is the first coordinate of the smallest solution to 
%\[
%\begin{aligned}
% \lambda_2 (s_a^2-s_d) + (\sigma+p(\lambda_1-\mu)) (s_a-s_d) + (\mu+(1-p)(\lambda_1-\mu))(1-s_a) &=0, \\
%\sigma (s_a-s_d) + \kappa \mu(1-s_d) &=0.  \\
%\end{aligned}
%\]
%Under Assumption~\ref{asp-simspont}, the smallest solution is
%\[ \Big( \frac{\lambda_1 \kappa \mu + \lambda_1(1-p)\sigma+ (p+\kappa)\mu \sigma \sigma}{\lambda_2(\kappa\mu+\sigma)},
%\frac{\kappa^2\lambda_2\mu^2+\kappa\lambda_1\mu\sigma+\kappa\lambda_2\mu\sigma+\lambda_1(1-p)\sigma^2+(p+\kappa)\mu\sigma^2}{\lambda_2(\kappa\mu+\sigma)^2} \Big).
%\]
   \end{enumerate}
Summarizing, our heuristics indicates that under condition \eqref{invasionpossible}, for large $K$, given that the mutants survive the first phase of invasion, the second and the third phase of the invasion are also successful with high probability.
%%%%with $\lambda_1>\mu$ is sufficient for the mutants to invade the population with positive probability. Here, the interesting sub-case is when $\lambda_2<\lambda_1$; else, it is known \color{red} reference? trivial? \color{black} that invasion is also possible without the mutants having a dormancy trait.
%In case $\lambda_2>\lambda_1$, invasion is still possible, but this would already be the case for $p=0$ (i.e., no switching to dormancy for the mutants), and hence the result is uninteresting. 
%%%%In case \eqref{invasionpossible} does not hold but $\lambda_1>\mu$, the mutants will die out before the resident population size leaves a small neighbourhood of its equilibrium, with a probability tending to one, since they are close to a subcritical or critical bi-type branching process.
%\begin{assumption}[Also spontaneous switching]\label{asp-simspont}
%The Markov process $(\mathbf N_t^K)_{t \geq 0}$ has rates according to \eqref{rates-compinduced}, modified in such a way that the transitions of the form $(n_1,n_{2a},n_{2d}) \to (n_1,n_{2a}-1/K,n_{2d}+1/K)$ happen at rate $K n_{2\alpha}(p\alpha(n_1+n_{2a})+\sigma$. The inequality \eqref{alsospontaneous} holds with $\lambda_1>\lambda_2>\mu$. 
%\end{assumption}
\subsection{Statement of results}\label{sec-statementofresult}
Recall that we have assumed $\lambda_1>\lambda_2>\mu>0$, and recall also the stable equilibrium $(\bar n_{2a}, \bar n_{2d})$, which is the unique solution of the system of equations \eqref{positiveequilibrium} under the assumption $\lambda_2>\mu$. For $\beta>0$ define
\[ S_\beta =  \{ 0 \} \times[ \bar n_{2a}-\beta, \bar n_{2a} + \beta ] \times [\bar n_{2d}-\beta, \bar n_{2d} + \beta], \numberthis\label{Sbetadef}\]
a stopping time at which $\mathbf N_t^K$ reaches this set:
\[ T_{S_\beta} : = \inf \{ t > 0 \colon \mathbf N_t^K \in S_\beta \}, \numberthis\label{TSbetadef} \]
and the first time when the rescaled mutant population size reaches a threshold $x \geq 0$ (from below or above):
\[ T_{x}^2 := \inf \{ t > 0 \colon KN_{2,t}^K = \lfloor x K \rfloor \}. \numberthis\label{Tlevel} \] %\color{red} We may need different versions of this stopping time, e.g., both actives and dormants reaching this level from above or from below, one of them hitting this level with the other one being controlled, etc. My guess is that for the first phase of the mutation \eqref{Tlevel} will be fine, and then we see what finer definition we need. \color{black} \\
We further note that the largest eigenvalue of the matrix $J$ defined in \eqref{Jdef} is given as follows.
\[ \widetilde \lambda = \frac{1}{2} \Big( (\lambda_2-\lambda_1-\kappa\mu-\sigma)+\sqrt{(\lambda_1-\lambda_2+\kappa\mu+\sigma)^2-4\big((\lambda_1-\lambda_2)(\kappa\mu+\sigma)-p(\lambda_1-\mu)\sigma\big)} \Big). \numberthis\label{lambdatildedef} \]
Our first main result characterizes the probability of mutant invasion in the large-population limit.
\begin{theorem}\label{thm-invasion}
Assume that \eqref{invasionpossible} holds. Assume further that
\[ N^K_1(0) \underset{K \to \infty}{\to} \bar n_1  \]
and 
\[ (N^K_{2a}(0),N^K_{2d}(0))=(\smfrac{1}{K},0). \]
Then for any $0<\beta<\min \{ \bar n_{2a}, \bar n_{2d} \}$, we have
\[ \lim_{K \to \infty} \mathbb P \Big( T_{S_\beta} < T_0^2 \Big) = 1- q. \]
\end{theorem}
Next, we identify the time of fixation of mutants in the case of a successful invasion.
\begin{theorem}\label{thm-success}
Under the assumptions of Theorem~\ref{thm-invasion}, we have that on the event $\{ T_{S_\beta} < T_0^2 \}$,
\[ \lim_{K \to \infty} \frac{T_{S_\beta}}{\log K} = \frac{1}{\widetilde\lambda}+\frac{1}{\mu+\alpha\bar n_{2a}-\lambda_1} \numberthis\label{invasion} \]
in probability.
\end{theorem}
Finally, we show that in case of an unsuccessful mutation, with high probability, the extinction takes a sub-logarithmic time (in particular, the extinction happens during the first phase of the invasion), and at the time of extinction the resident population is close to its equilibrium population size.
\begin{theorem}\label{thm-failure}
Under the assumptions of Theorem~\ref{thm-invasion}, we have that on the event $\{ T_0^2 < T_{S_\beta}\}$, 
\[ \lim_{K \to \infty} \frac{T_0^2}{\log K} =0 \numberthis\label{extinction} \]
and
\[ \mathds 1 \{ T_{S_\beta} > T_0^2 \} \Big\vert \mathbf N^K_{T_0^2}-(\bar n_1,0,0) \Big\vert \underset{K \to \infty}{\longrightarrow} 0, \numberthis\label{lastoftheorem} \]
both in probability.
\end{theorem}
The proof of Theorems~\ref{thm-invasion}. \ref{thm-success}, and \ref{thm-failure} will be carried out in Section~\ref{sec-proof}. In multiple parts of the proof, we are able to employ arguments that are similar to the ones used in \cite{C+19, C+16} for the three phases of invasion in individual-based models in the context of emergence of homogamy, respectively speciation. A particular additional difficulty of our setting lies in guaranteeing convergence of the underlying dynamical system \eqref{3dimlotkavolterra} to its stable equilibrium $(0,\bar n_{2a},\bar n_{2d})$, in other words, in verifying certain global attractor properties of this equilibrium. Here, none of the methods of the two aforementioned papers are applicable (see the proof of Lemmas~\ref{lemma-2dODE} and~\ref{lemma-3dODE}). 
Our dynamical system is rather different from the ones considered in \cite{C+19, C+16}, which have stronger monotonicity properties but also exhibit non-hyperbolic equilibria. The lack of monotonicity in our system is due to the switches between activity and dormancy and to the fact that dormant individuals are not affected by competition. These differences also influence other parts of the proof of our main theorems nontrivially (see e.g.~the construction of the couplings in the proofs of Propositions~\ref{prop-firstphase} and \ref{prop-thirdphase}). 

\section{Discussion}\label{sec-discussion}
This section touches the following topics. In Section~\ref{sec-invasionpossible} we provide an interpretation of condition \eqref{invasionpossible} that is crucial for our main results and comment on the notion of {\em invasion fitness}. The relevance of competition-induced vs.\ spontaneous switching is discussed in Section~\ref{sec-discussion-compinduced}, and the case where the first mutant individual is initially dormant instead of active is discussed in Section~\ref{sec-discussion-startingwithadormant}. In Section~\ref{sec-biologists!} we comment on potential experimental studies related to the subject of this paper for model verification.
\subsection{Interpretation of the condition of the theorems}\label{sec-invasionpossible}
Condition \eqref{invasionpossible} is equivalent to the assertion that the advantage of residents caused by their higher birth rate is less than the advantage of the mutants caused by their ability to become dormant under competitive pressure, at the beginning of the invasion where the mutants are rare. Indeed, the right-hand side of \eqref{invasionpossible} equals the rate at which those active mutant individuals move to the seed bank that afterwards become active again before dying. Indeed, active mutants become dormant at rate $p \alpha \bar n_1$, and given that they have become dormant, the probability that they turn active again (instead of dying in the seed bank) is $\frac{\sigma}{\kappa\mu+\sigma}$. In the case $\kappa=0$ of no death in the seed bank, \eqref{invasionpossible} reduces to 
    \[ \frac{\lambda_1-\mu}{1} < \frac{\lambda_2-\mu}{1-p}, \]
    %That is, invasion may be possible with $\lambda_2<\lambda_1$, but difference between the birth rate and death rate for the residents (which is by default positive) should be at most $1/(1-p)$ times the same difference for the mutants. Note that $\lambda_2<\lambda_1$ is also possible for $\kappa>0$ as long as \eqref{invasionpossible} holds. \\
where $1-p$ is the probability that a mutant affected by a competitive event dies. On the complementary event, this mutant will eventually become active again.
%else this mutant will almost surely become active before dying.

Note that $\lambda_2>\mu$ automatically follows from \eqref{invasionpossible} given that $\lambda_1>\mu$. Thus, our model is free from evolutionary suicide: mutants who are not able to survive on their own will not make the resident population go extinct with asymptotically positive probability.

The \emph{invasion fitness} is the exponential
growth rate of a mutant born with a given trait in the presence of the current
equilibrium population \cite[Section 1.3.2]{B19}. In the present setting, the precise formulation of such a quantity is not immediate, for the following reasons. First, the total mutant population size process $(K N_{2,t}^K)_{t \geq 0}$ is not Markovian and hence has no well-defined exponential rate. Second, the pair of active and dormant coordinates $((KN_{2a,t}^K, KN_{2d,t}^K))_{t \geq 0}$ is Markovian, but its initial growth rate depends delicately on the initial condition. More precisely, for $\kappa>0$, the mutant population has a lower probability to survive if it starts with one dormant and no active individual than if it starts with one active and no dormant one (see Section~\ref{sec-discussion-startingwithadormant} for further details). Nevertheless, if we define the invasion fitness as the principal eigenvalue (a.k.a.~Lyapunov exponent) $\widetilde\lambda$ of the mean matrix $J$, then this eigenvalue is positive if and only if the condition~\eqref{invasionpossible} holds, in other words, it has the same sign as the expression
\[ p(\lambda_1-\mu)\frac{\sigma}{\kappa\mu+\sigma} - \lambda_1 + \lambda_2. \]
This sign is positive (respectively zero or negative) if and only if the approximating branching process $((\widehat Z_{2a}(t),\widehat Z_{2d}(t))_{t \geq 0}$ is supercritical (respectively critical or subcritical). 
Further, according to~\cite[Section 7]{AN72}, $\widetilde\lambda$ is equal to the mean growth rate of the approximating branching process $(\widehat Z_{2a}(t),\widehat Z_{2d}(t))$, which makes it rightful to call this eigenvalue the invasion fitness. 
\subsection{A comparison between spontaneous and competition-induced switching, and the case without dormancy trait}\label{sec-discussion-compinduced}
We have seen that the bi-type mutant population is able to survive on its own if $\lambda_2>\mu$, and if \eqref{invasionpossible} holds, then the mutants will invade the population with positive probability even if $\lambda_2<\lambda_1$. Let us note that without the mutants having a dormancy trait (i.e., for $p=0$), even though mutants can still survive on their own as soon as $\lambda_2>\mu$, invasion is not possible as long as $\lambda_2 \leq \lambda_1$. This is true because the approximating branching process is not supercritical in this case. 

For $\kappa>0$, it is not even the case that mutants are fit on their own if the switching from activity to dormancy is not competition-induced but \emph{spontaneous}, i.e., if an active mutant individual switches to dormancy at some fixed rate $\sigma'>0$. There, in absence of residents, for large $K$, the rescaled mutant population is approximated by the system of ODEs
\[
\begin{aligned}
\dot n_{2a}(t) &  = n_{2a}(t)(\lambda_2-\mu-\alpha n_{2a}(t)-\sigma')+\sigma n_{2d}(t),\\
\dot n_{2d}(t) & = \sigma' n_{2a}(t) - (\kappa \mu+\sigma) n_{2d}(t).
\end{aligned} \numberthis\label{2dimlotkavolterra}
\]
Hence, the origin is asymptotically stable if and only if $(\lambda_2-\mu-\sigma')(-\kappa\mu-\sigma)-\sigma\sigma'<0$, i.e., 
\[ \lambda_2 < \mu + \frac{\kappa \mu\sigma'}{\kappa \mu+ \sigma}. \numberthis\label{effectivedeathrate} \]
I.e., there are values $\lambda_2>\mu$ such that the mutant population dies out with high probability if $K \to \infty$. The right-hand side of \eqref{effectivedeathrate} is the \emph{effective death rate}: indeed, an active individual dies at rate $\mu$, but additionally at rate $\sigma'$ it becomes dormant, where it dies with probability $\frac{\kappa\mu}{\kappa\mu+\sigma}$ before ever becoming active (and capable of reproducing) again.

In the case of spontaneous switching, it is easy to show that the matrix defined analogously to $J$ (cf.~\eqref{Jdef}) has no positive eigenvalue for $\lambda_2<\lambda_1$. I.e., mutant invasion is only possible if the birth rate of mutants is higher than the one of the residents.
  
We expect that in case both spontaneous and competition-induced switching are present in the model, the behaviour of the system remains similar to the case of purely competition-induced switching, however, with a higher effective death rate, and hence condition \eqref{invasionpossible} is not satisfactory for invasion; $\lambda_2$ has to satisfy a stronger condition, which can be derived similarly to \eqref{invasionpossible}. In order to keep the notation simple, we do not consider this case of combined switching in the present paper.
\subsection{Starting with one dormant individual}\label{sec-discussion-startingwithadormant}

Let us recall that $s_a$ is the extinction probability of the approximating bi-type branching process $((\widehat Z_{2a}(t),\widehat Z_{2d}(t))_{t \geq 0}$ starting from $(1,0)$, and $s_d$ the same probability starting from $(0,1)$. Note that the second equation of \eqref{extinctionequation} reads as
\[ s_d = \frac{\sigma s_a + \kappa \mu}{\kappa \mu+\sigma}. \numberthis\label{sdsa} \]
Note that for $\kappa=0$, \eqref{sdsa} reads as $s_d=s_a$. Thanks to the Markov property of our population process, \eqref{sdsa} can be interpreted as follows: given that $(\widehat Z_{2a}(0), \widehat Z_{2d}(0))=(0,1)$, with probability $\smfrac{\kappa\mu}{\kappa\mu+\sigma}$ the process dies out immediately at the first jump time that affects this single dormant individual. Else (i.e., with probability $\smfrac{\sigma}{\kappa\mu+\sigma}$), it jumps to $(1,0)$, where it has probability $s_a$ to die out. This argumentation also implies the following. Let $T_{1,0}$ be the expected extinction time of the mutant population starting from $(1,0)$ and $T_{0,1}$ the same starting from $(0,1)$. Then we have
\[ \E\big[T_{0,1}\mathds 1 \{ T_{0,1} <\infty \} \big] = \frac{1}{\kappa\mu+\sigma}+\frac{\sigma}{\kappa\mu+\sigma} \E\big[T_{1,0}\mathds 1 \{ T_{1,0}<\infty \}\big], \]
where $\frac{1}{\kappa\mu+\sigma}$ is the expected time of the first jump of the Markov chain. Hence, extinction probabilities and extinction times started from $(0,1)$ can easily be handled using the same quantities started from $(1,0)$. This is why our main results describe only the latter case.

\subsection{Experimental studies}\label{sec-biologists!}
It would be highly interesting to check the results of the present paper experimentally. In the spirit of the mathematical analysis of the Lenski experiment \cite{GKWY16,LT94} (that also exhibits the three phases of adaptive dynamics invasion),
one could think of setting up a controlled experiment where the environment is kept constant over time, with a relatively high but fixed amount of resources. Now, one would need to find two types of microorganisms such that both of them are able to survive on their own in this environment, but the first type reproduces faster, whereas only the second one has a dormancy trait, in such a way that condition \eqref{invasionpossible} holds for the parameters estimated in the experiment. Then, one would first have to establish a resident population of the first type, then augment it by a single individual (or several individuals) of the second type, and continue the experiment until one of the types becomes extinct. Repeating this experiment several times, it would become apparent whether the invasion of the second type has a positive probability, and whether the invasion probability would come close to the one predicted by our model.

Certainly, the model presented in this paper captures only a small number of features of natural populations. Hence, scenarios excluded by our model such as coexistence of the two types may occur in the experiment. This could lead to interesting feed-back and theoretical  model extensions.
%  Possibly, the model has to be modified in order to make it fit to the experiment. However, the fact that both populations can survive on their own should imply that at the time of extinction of one type, the other population should be close to its equilibrium size with high probability.

\section{Proofs}\label{sec-proof}
This section is split into four parts: Section~\ref{sec-firstphase} investigates the first phase of the invasion: the growth or extinction of the mutants. The next two phases only occur if the mutants survive the first phase. Section~\ref{sec-secondphase} deals with the second phase, where the rescaled population size process is approximated by the system of ODEs \eqref{3dimlotkavolterra}, and Section~\ref{sec-thirdphase} describes the third phase where the resident population dies out. Using all these, we complete the proof of our theorems in Section~\ref{sec-proofremainder}. Throughout the proof we will assume that $\beta \in (0, \min \{ \bar n_{2a}, \bar n_{2d} \})$. 

\subsection{The first phase of invasion: growth or extinction of the mutant population}\label{sec-firstphase}
The analysis of this phase proceeds similarly to \cite[Section 3.1]{C+19}. However, the presence of dormancy induces nontrivial changes in some coupling arguments (see e.g.~the construction of the coupled process appearing in \eqref{upsilondef}). On the other hand, since we have a monomorphic resident population, some arguments can be simplified or omitted, and the order of proof ingredients will change accordingly.

We now define additional stopping times that will be relevant for this phase. The first one is the time when the resident population first leaves a small-neighbourhood of its equilibrium: for any $\eps>0$,
\[ R_\eps: = \inf \Big\{ t \geq 0 \colon \big| N_{1,t}^K - \bar n_1 \big| > \eps \Big\}. \]
Then our goal is to verify the following proposition.
\begin{prop}\label{prop-firstphase}
Assume that \eqref{invasionpossible} holds with $\lambda_1>\lambda_2>\mu$. Let $K \mapsto m_1^K$ be a function from $(0,\infty)$ to $[0,\infty)$ such that $m_1^K \in \smfrac{1}{K} \N_0$ and $\lim_{K \to \infty} m_1^K=\bar n_1$. Then there exists a function $f \colon (0,\infty) \to (0,\infty)$ tending to zero as $\eps \downarrow 0$ such that for any $\xi \in [1/2,1]$,
\[ \limsup_{K \to \infty} \Big| \P \Big( T_{\eps^\xi}^2 < T_0^2 \wedge R_{2 \eps}, \Big| \frac{T_{\eps^\xi}^2}{\log K} - \frac{1}{\widetilde \lambda} \Big| \leq f(\eps)~ \Big| \mathbf N_0^K=\big(m_1^K,\frac{1}{K},0\big)  \Big)-(1-q) \Big| = o_\eps(1) \numberthis\label{invasiontime} \]
and
\[ \limsup_{K \to \infty} \Big| \P\Big(T_0^2 < T_{\eps^\xi}^2 \wedge R_{2\eps}~ \Big|  \mathbf N_0^K=\big(m_1^K,\frac{1}{K},0\big)  \Big) - q \Big|=o_{\eps}(1), \numberthis\label{secondofprop} \]
where $o_{\eps}(1)$ tends to zero as $\eps \downarrow 0$.
\end{prop}
In order to prove the proposition, we first verify the following lemma.
\begin{lemma}\label{lemma-residentsstay}
Under the assumptions of Proposition~\ref{prop-firstphase}, there exists a positive constant $\eps_0$ such that for any $\xi \in [1/2,1]$ and $0<\eps \leq \eps_0$,
\[ \limsup_{K \to \infty} \P \big(R_{2\eps} \leq T^2_{\eps^\xi} \wedge T_0^2 \big)=0. \]
\end{lemma} 
\begin{proof}
We verify this lemma via coupling the rescaled population size $N_{1,t}^K$ with two birth-and-death processes, $Y_{1,t}^1$ and $Y_{1,t}^2$, on time scales where the mutant population is still small compared to $K$. More precisely, following \cite[Section 3.1.2]{C+19},
\[ Y_{1,t}^1 \leq N_{1,t}^K \leq Y_{1,t}^2, \qquad \text{a.s.} \qquad \forall t \leq T_0^2 \wedge T_{\eps^\xi}^2. \numberthis\label{firstresidentcoupling} \]
The latter processes will also depend on $K$, but we omit the notation $K$ from their nomenclature for simplicity. In order to satisfy \eqref{firstresidentcoupling}, the processes $Y_1^1=(Y_{1,t}^1)_{t \geq 0}$ and $Y_1^2=(Y_{1,t}^2)_{t \geq 0}$ can be chosen with the following birth and death rates
\begin{align*}
Y_{1,t}^1 \colon \quad & \frac{i}{K} \to \frac{i+1}{K} \quad \text{at rate} \quad i\lambda_1, \\
                 & \frac{i}{K} \to \frac{i-1}{K} \quad \text{at rate} \quad i \big( \mu+\alpha \frac{i}{K} + \alpha \eps^\xi \big).
\end{align*}
and
\begin{align*}
Y_{1,t}^2 \colon  \quad & \frac{i}{K} \to \frac{i+1}{K} \quad \text{at rate} \quad i\lambda_1, \\
                 & \frac{i}{K} \to \frac{i-1}{K} \quad \text{at rate} \quad i \big( \mu+\alpha \frac{i}{K} \big).
\end{align*}
Let us estimate the time until which the processes $Y_1^1$ and $Y_2^1$ stay close to the value $\bar n_1$. We define the stopping times
\[ R^i_{\eps}:=\inf \big\{ t \geq 0 \colon Y_{1,t}^i \notin [\bar n_1-\eps,\bar n_1+\eps]\big\}, \qquad i \in \{1,2\},~\eps>0. \]
For large $K$, according to \cite[Theorem 2.1, p.~456]{EK}, the dynamics of $Y_{1,t}^1$ is close to the one of the unique solution to
\[ \dot n = n(\lambda_1-\mu-\alpha n-\alpha \eps^{\xi}). \]
The equilibria of this ODE are $0$ and $\bar n_1^{(\eps)}=\smfrac{\lambda_1-\mu-\alpha \eps^{\xi}}{\alpha}=\bar n_1-\eps^{\xi}$. Since $\lambda_1>\mu$, the latter equilibrium is positive for all sufficiently small $\eps>0$. Linearizing implies that for all small enough $\eps>0$ (namely, for $\eps$ such that $\alpha \eps^\xi < \lambda_1-\mu$), the equilibrium $0$ is unstable and the one $\bar n_1^{(\eps)}$ is asymptotically stable. A direct analysis of the sign of $n(\lambda_1-\mu-\alpha n-\alpha \eps^{\xi})$ implies that for such $\eps$, any solution with a positive initial condition converges to the stable equilibrium $\bar n_1^{(\eps)}$ as $t \to \infty$. These also imply that there exists $\eps_0 >0$ such that for all $0 <\eps \leq \eps_0$,
\[ \big|\bar n_1-\bar n_1^{(\eps)}\big| = \eps^{\xi} \qquad \text{ and } \qquad 0 \notin [\bar n_1-2 \eps, \bar n_1+2 \eps]. \]
%This in particular holds for all $\eps \in (0,\eps_0)$ in case $\eps_0 > 0$ is such that 
%\[  |\bar n_1^{(\eps)}-\bar n_1| \leq \eps^\xi \qquad \text{and} \qquad 0 \notin [\bar n_1-2\eps^\xi,\bar n_1+2\eps^\xi ]. \]
%Hence, for such $\eps$, any solution with positive initial condition converges to the stable equilibrium as $t \to \infty$. 
Now, using a result about exit of jump processes from a domain by Freidlin and Wentzell \cite[Chapter 5]{FW84}, %analogously to \cite[Section 3.1.2]{C+19}, 
there exists a family (over $K$) of Markov jump processes $\widetilde Y^1_1=(\widetilde Y^1_{1,t})_{t \geq 0}$ whose transition rates are positive, bounded, Lipschitz continuous, and uniformly bounded away from 0 such that for
\[ \widetilde R_\eps^1 = \inf \big\{ t \geq 0 \colon \widetilde Y^1_{1,t} \notin [\bar n_1-\eps,\bar n_1+\eps]\big\}, \qquad i \in \{1,2\},~\eps>0, \]
there exists $V>0$ such that
\[ \P(R^1_{2\eps}>\e^{KV}) = \P(\widetilde R^1_{2\eps}>\e^{KV}) \underset{K \to \infty}{\longrightarrow} 0. \numberthis\label{FW1} \]
Using similar arguments for $N_{1}^2$, we derive that for $\eps>0$, $V>0$ small enough, we have that
\[ \P(R^1_{2\eps}>\e^{KV})\underset{K \to \infty}{\longrightarrow} 0.  \numberthis\label{FW2} \]
%\color{red} As far as I see, with $Y_{1,t}^2$ things are even easier because the corresponding ODE has stable equilibrium equal to $\bar n_1$. \color{black} 
Now, on the event $\{ R_{2\eps} \leq T_0^2 \wedge T_{\eps^\xi}^2 \}$ we have $R_{2\eps} \geq R_{2\eps}^1 \wedge R_{2\eps}^2$. Using \eqref{FW1} and \eqref{FW2}, we derive that
\[ \limsup_{K \to \infty} \P \big( R_{2\eps} \leq \e^{KV}, R_{2\eps} \leq T_0^2 \wedge T_{\eps^\xi}^2 \big) = 0. \]
Moreover, using Markov's inequality,
\begin{align*}
    \P(R_{2\eps} \leq T_0^2 \wedge T_{\eps^\xi}^2) & \leq \P \big( R_{2\eps} \leq \e^{KV}, R_{2\eps} \leq T_0^2 \wedge T_{\eps^\xi}^2 \big) + \P(R_{2\eps} \wedge T_0^2 \wedge T_{\eps^\xi}^2 \geq \e^{KV} \big) \\ &  \leq \P \big( R_{2\eps} \leq \e^{KV}, R_{2\eps} \leq T_0^2 \wedge T_{\eps^\xi}^2 \big) + \e^{-KV} \E(R_{2\eps} \wedge T_0^2 \wedge T_{\eps^\xi}^2 ). 
\end{align*}
Since we have
\[ \E \big[R_{2\eps} \wedge T_0^2 \wedge T_{\eps^\xi}^2 \big] \leq \E \Big[\int_0^{R_{2\eps} \wedge T_0^2 \wedge T_{\eps^\xi}^2} K N_{2,t}^K \d t \Big], \]
it suffices to show that there exists $C>0$ such that 
\[ \E \Big[\int_0^{R_{2\eps} \wedge T_0^2 \wedge T_{\eps^\xi}^2} K N_{2,t}^K \d t \Big] \leq C \eps^\xi K. \numberthis\label{expectedstoppingtime} \]
This can be done similarly to \cite[Section 3.1.2]{C+19}. %\color{red} , here we temporarily leave the proof of \cite[Lemma 3.2]{C+19} and do some parts of the proof of \cite[Lemma 3.1]{C+19}, as easy as possible  \color{black}.
Indeed, let $\Lcal$ be the infinitesimal generator of $(\mathbf N^K_t)_{t \geq 0}$. We want to show that there exists a function $g \colon (\smfrac{1}{K} \N_0)^3 \to \R$ defined as
\[ g(n_{1},n_{2a},n_{2d})=\gamma_1 n_{2a} + \gamma_2 n_{2d} \numberthis\label{gamma1gamma2} \]
such that
\[ \Lcal g(\mathbf N_{t}^K) \geq N_{2,t}^K. \numberthis\label{largegenerator}\]
If \eqref{largegenerator} holds, then \eqref{expectedstoppingtime} follows because thanks to Dynkin's formula,
\[ 
\begin{aligned}
\E \Big[\int_0^{R_{2\eps} \wedge T_0^2 \wedge T_{\eps^\xi}^2} K N_{2,t}^K \d t \Big] & \leq  \E \Big[\int_0^{R_{2\eps} \wedge T_0^2 \wedge T_{\eps^\xi}^2} K\Lcal g(\mathbf N_{t}^K)\d t \Big] =\E\big[ K g(\mathbf N_{R_{2\eps} \wedge T_0^2 \wedge T_{\eps^\xi}^2}^K)-Kg(\mathbf N_{0}^K) \big] \\
& \leq (\gamma_1 \vee \gamma_2) \eps^\xi K - ( \gamma_1 \wedge \gamma_2 ),
\end{aligned}
\]
which implies \eqref{expectedstoppingtime}, independently of the signs of $\gamma_1$ and $\gamma_2$. Let us apply the infinitesimal generator $\mathcal L$ to the function $g$ defined in \eqref{gamma1gamma2}. We obtain 
\[ \Lcal g(\mathbf N_t^K) = N_{2a,t}^K \big[ (\lambda_2-\mu-\alpha (N_{1,t}^K+N_{2a,t}^K))\gamma_1 +p\alpha (N_{1,t}^K+N_{2a,t}^K) \gamma_2 \big] + N_{2d,t}^K \big[ \sigma  \gamma_1 -(\kappa\mu+\sigma)  \gamma_2 \big]. \] 
Hence, according to \eqref{largegenerator}, it sufficies to show that there exists $\gamma_1,\gamma_2 \in \R$ such that the following system of inequalities is satisfied:
\begin{align}
    (\lambda_2-\mu-\alpha (N_{1,t}^K+N_{2a,t}^K))\gamma_1 +p\alpha (N_{1,t}^K+N_{2a,t}^K) \gamma_2 & >1, \label{first-gamma1gamma2} \\
     \sigma \gamma_1 -(\kappa\mu+\sigma)\gamma_2 & >1. \label{second-gamma1gamma2}
\end{align}
Since $N_{1,t}^K+N_{2a,t}^K$ varies in $t$, the system \eqref{first-gamma1gamma2}--\eqref{second-gamma1gamma2} of inequalities is not easy to handle. However, for $t \in [0,R_{2\eps} \wedge T_0^2 \wedge T_{\eps^\xi}^2]$, we have $\alpha N_t^K \leq \alpha (\bar n_1+2\eps+\eps^\xi)$ and $p\alpha N_t^K \geq p\alpha (\bar n_1-2\eps)$. Hence, 
\[ (\lambda_2-\mu-\alpha (N_{1,t}^K+N_{2a,t}^K))\gamma_1 +p\alpha (N_{1,t}^K+N_{2a,t}^K) \gamma_2  \geq (\lambda_2-\mu-\alpha (\bar n_1+2\eps+\eps^\xi)) \gamma_1 + p \alpha (\bar n_1-2\eps)\gamma_2, \]
which implies that \eqref{first-gamma1gamma2} is satisfied as soon as
\[  (\lambda_2-\mu-\alpha (\bar n_1+2\eps+\eps^\xi)) \gamma_1 + p \alpha (\bar n_1-2\eps)\gamma_2>1, \] 
which, according to the definition of $\bar n_1$, can also written as
\[ (\lambda_2-\lambda_1-2\eps-\eps^\xi)\gamma_1 + p(\lambda_1-\mu-2\eps) \gamma_2 >1. \numberthis\label{firststronger-gamma1gamma2} \]
Let us verify the existence of $\gamma_1$ and $\gamma_2$ satisfying \eqref{firststronger-gamma1gamma2} and \eqref{second-gamma1gamma2}. First of all, we can rewrite \eqref{second-gamma1gamma2} as follows
\[ \gamma_2 < \frac{\sigma \gamma_1-1}{\kappa\mu+\sigma}. \numberthis\label{gamma2rewritten} \]
Hence, let us first consider the equation
\[ (\lambda_2-\lambda_1-2\eps-\eps^\xi)\gamma_1 + p(\lambda_1-\mu-2\eps) \frac{\sigma \gamma_1-1}{\kappa\mu+\sigma} > 1. \numberthis\label{withgamma2bound} \]
The inequality \eqref{invasionpossible} is satisfied by assumption, and hence there exists $\eps>0$ such that
\[ (\lambda_2-\lambda_1-2\eps-\eps^\xi) + p(\lambda_1-\mu-2\eps) \frac{\sigma}{\kappa\mu+\sigma} > 0. \]
Hence, $(\lambda_2-\lambda_1-2\eps-\eps^\xi)\gamma_1 + p(\lambda_1-\mu-2\eps) \frac{\sigma\gamma_1}{\kappa\mu+\sigma}$ tends to infinity as $\gamma_1 \to \infty$, in particular, for all sufficiently large $\gamma_1$ it is strictly larger than $1+\smfrac{p(\lambda_1-\mu-2\eps)}{\kappa\mu+\sigma}$, and thus \eqref{withgamma2bound} holds. By continuity of the function $x \mapsto p(\lambda_1-\mu-2\eps) x$, this implies that for any $\gamma_1$ satisfying \eqref{withgamma2bound} there exists $\gamma_2$ satisfying \eqref{gamma2rewritten} such that \eqref{firststronger-gamma1gamma2} holds. We conclude the lemma.
\end{proof}
\begin{proof}[Proof of Proposition~\ref{prop-firstphase}.] In what follows, we consider our population process on the event
\[ A_\eps := \{ T_0^2 \wedge T_{\eps^\xi}^2 < R_{2\eps} \} \]
for sufficiently small $\eps>0$. On this event, the invasion or extinction of the mutant population will happen before the resident population substantially deviates from its equilibrium size. We couple on $A_\eps$ the process $(KN_{2a,t}^K,KN_{2d,t}^K)$ with two bi-type branching processes $(Z_{2a,t}^{\eps,-}, Z_{2d,t}^{\eps,-})$ and $(Z_{2a,t}^{\eps,+}, Z_{2d,t}^{\eps,+})$ on $\N_0^2$ (which again depend on $K$, but we omit that from the notation for readability) such that almost surely, for any $t <  t_\eps:=T_0^2 \wedge T_{\eps^\xi}^2 \wedge R_{2\eps}$ and $\upsilon \in \{ a,d \}$,
\[ \begin{aligned}
Z_{2\upsilon,t}^{\eps,-} & \leq \widehat Z_{2\upsilon}(t) \leq Z_{2\upsilon,t}^{\eps,+}, \\
Z_{2\upsilon,t}^{\eps,-} & \leq K N_{2\upsilon,t}^K \leq Z_{2\upsilon,t}^{\eps,+},
\end{aligned} \numberthis\label{upsilondef}
\]
where we recall the approximating branching process $(\widehat Z_{2a}(t),\widehat Z_{2d}(t))$ defined in Section~\ref{sec-invasionconditions}.
We claim that in order to satisfy \eqref{upsilondef}, these processes can be defined with the following jump rates:
\[ \begin{aligned}
(Z_{2a,t}^{\eps,-},Z_{2d,t}^{\eps,-}) \colon \quad & (i,j) \to (i+1,j) & \text{at rate} & \quad i \lambda_2, \\
& (i,j) \to (i-1,j)& \text{at rate}  &\quad i(\mu+(1-p)\alpha(\eps^\xi+\bar n_1+2\eps)+p\alpha(4\eps+\eps^\xi)), \\
& (i,j) \to (i-1,j+1) & \text{at rate}& \quad ip\alpha(\bar n_1-2\eps), \\
& (i,j) \to (i+1,j-1) & \text{at rate} &\quad j\sigma, \\
& (i,j) \to (i,j-1) & \text{at rate} &\quad j\kappa\mu,
\end{aligned}
\]
and
\[ \begin{aligned}
(Z_{2a,t}^{\eps,+},Z_{2d,t}^{\eps,+}) \colon \quad & (i,j) \to (i+1,j) & \text{at rate} & \quad i \lambda_2, \\
& (i,j) \to (i-1,j)& \text{at rate}  &\quad i(\mu+(1-p)\alpha(\bar n_1-2\eps)-p\alpha(4\eps+\eps^\xi)), \\
& (i,j) \to (i-1,j+1) & \text{at rate}& \quad ip\alpha(\bar n_1+2\eps+\eps^\xi), \\
& (i,j) \to (i+1,j-1) & \text{at rate} &\quad j\sigma, \\
& (i,j) \to (i,j-1) & \text{at rate} &\quad j\kappa\mu.
\end{aligned}
\]
Informally speaking, \eqref{upsilondef} holds thanks to the fact that for branching processes having the same kind of transitions as $(KN_{2a,t}^K,KN_{2d,t}^K)_{t \geq 0}$, competition-induced switching to dormancy is more favourable for an active mutant individual than immediate death by competition, but not better, and for $\kappa>0$ strictly worse, than not being hit by a competitive event at all.

Now, for $\diamond \in\{+,-\}$ and for a fixed initial condition $(i,j)$, the total competitive event rate of $(Z_{2a,t}^{\eps,\diamond},Z_{2d,t}^{\eps,\diamond})$ is given as the sum of the $(i,j) \to (i-1,j)$ and the $(i,j) \to (i-1,j+1)$  jump rate corresponding to the process. Given that a competitive event has happened, the ratio of the probability of death by competition and the one of switching to dormancy is equal to the ratio of the $(i,j) \to (i-1,j)$ rate and the $(i,j) \to (i-1,j+1)$ rate. Further, for any fixed initial condition, $(Z_{2a,t}^{\eps,-},Z_{2d,t}^{\eps,-})$ has higher death rate, higher total competitive event rate, but lower rate for active$\to$dormant switching rate than $(\widehat Z_{2a}(t),\widehat Z_{2d}(t))$ for any $t \geq 0$ or than $(KN_{2a,t}^K,KN_{2d,t}^K)$ for $t < t_\eps$ on the event $A_\eps$, while all other rates are the same for all these processes. Birth-and-death processes are coordinatewise nonincreasing (for $\kappa>0$ decreasing) in the rate of competitive events. Indeed, after a competitive event the affected active mutant individual either dies immediately or moves to the seed bank, where it dies with probability less than one (but for $\kappa>0$ more than zero) before ever becoming active again. On the other hand, the $(i,j)\to(i-1,j+1)$ switching rates are the lowest for $(Z_{2a,t}^{\eps,-},Z_{2d,t}^{\eps,-})$, which ensures that less individuals enter the seed bank and the couplings \eqref{upsilondef} hold also for $\upsilon=d$. The corresponding inequalities for $(Z_{2a,t}^{\eps,+},Z_{2d,t}^{\eps,+})$ in \eqref{upsilondef} follow similarly since this process has the lowest rate for death and for competitive events in total but the highest rate for active$\to$dormant switching.

For $\diamond \in \{ +, - \}$, let $q^{(\eps,\diamond)}$ denote the extinction probability of the process $(Z_{2a,t}^{\eps,\diamond}, Z_{2d,t}^{\eps,\diamond})$ started from $(1,0)$. The extinction probability of a supercritical branching process is continuous with respect to all kinds of transitions that the mutant population in our model has. Given the total competitive event rate, this probability increases with the rate of active death by competition. Further, given the ratio between the rate of death by competition and the one of active$\to$dormant switching, it increases with the total competitive event rate. These assertions can be proven using the methods of \cite[Sections A.3]{C+19}. Hence, by the first line of \eqref{upsilondef}, we have
$q^{(\eps,+)} \leq q \leq q^{(\eps,-)}$ for fixed $\eps>0$ and
\[ 0 \leq \liminf_{\eps \downarrow 0} \big| q^{(\eps,\diamond)}-q \big| \leq \limsup_{\eps \downarrow 0} \big| q^{(\eps,\diamond)}-q \big| \leq \limsup_{\eps \downarrow 0} \big| q^{(\eps,-)}-q^{(\eps,+)} \big| = 0, \qquad  \forall \diamond \in \{ +, - \}, \numberthis\label{qineq} \]
where we recall the extinction probability $q$ defined in \eqref{qdef}.  

Next, we prove that the probabilities of extinction and invasion of the actual process $(N_{2a,t}^K,N_{2d,t}^K)$ also tend to $q$ and $1-q$, respectively, with high probability as $K \to \infty$. We define the stopping times
\[ T_x^{(\eps,\diamond),2} := \inf \{ t >0 \colon Z^{(\eps,\diamond)}_{2a}(t) +Z^{(\eps,\diamond)}_{2d}(t)= \lfloor K x \rfloor \}, \qquad \diamond \in \{ +, - \}, x \in \R. \]
Thanks to the coupling in the second line of \eqref{upsilondef}, which is valid on $A_\eps$, we have
\[ \P\big(T_{\eps^\xi}^{(\eps,-),2} \leq T_0^{(\eps,-),2}, A_\eps\big) \leq \P\big(T_{\eps^\xi}^{2} \leq T_0^{2},A_\eps \big) \leq  \P\big(T_{\eps^\xi}^{(\eps,+),2} \leq T_0^{(\eps,+),2}, A_\eps\big) \numberthis\label{couplednonextinction} \] 
Indeed, if a process reaches the size $K\eps^\xi$ before dying out, then the same holds for a larger process. However, $A_\eps$ is independent of $(Z_{2a,t}^{\eps,\diamond}, Z_{2d,t}^{\eps,\diamond})$ for both $\diamond = +$ and $\diamond = -$, and hence 
\[ \liminf_{K \to \infty} \P\big( T_{\eps^\xi}^{(\eps,-),2}\leq T_0^{(\eps,-),2}, A_\eps \big)=\liminf_{K \to \infty}   \P(A_\eps)\P\big( T_{\eps^\xi}^{(\eps,-),2} \leq T_0^{(\eps,-),2})  \geq (1-q^{(\eps,-)})(1-o_\eps(1)) \numberthis\label{eps-LB} \]
and
\[ \limsup_{K \to \infty}  \P\big( T_{\eps^\xi}^{(\eps,+),2}\leq T_0^{(\eps,+),2}, A_\eps \big) =\limsup_{K \to \infty}   \P(A_\eps) \P\big( T_{\eps^\xi}^{(\eps,+),2} \leq T_0^{(\eps,+),2})  \leq (1-q^{(\eps,+)})(1+o_\eps(1)). \numberthis\label{eps+UB} \]
Letting $K \to \infty$ in \eqref{couplednonextinction} and applying \eqref{eps-LB} and \eqref{eps+UB} yields that %any limit point $p_0^{\eps}$ of $\P(T_{\eps^\xi}^{2} \leq T_0^{2},A_\eps )$ satisfies
%\[ (1-q^{(\eps,-)})(1-o_\eps(1)) \leq p_0^{\eps} \leq (1-q^{(\eps,+)})(1+o_\eps(1)). \]
%Hence,
\[ \limsup_{K \to \infty} \big| \P(T_{\eps^\xi}^{2} \leq T_0^{2},A_\eps )-(1-q) \big|=o_\eps(1), \]
as required. The equation \eqref{secondofprop} can be derived similarly. 

It remains to show that in the case of invasion, 
%(which happens with probability tending to $1-q$) 
the time before reaching size $K \eps^\xi$ is of order $\log K/\widetilde \lambda$, where $\widetilde \lambda$ was defined in \eqref{lambdatildedef} as the maximal eigenvalue of the matrix $J$ defined in \eqref{Jdef}, which is positive under our assumptions. 

Let $\widetilde\lambda^{(\eps,\diamond)}$, $\diamond \in \{ +, - \}$, denote the maximal eigenvalue of the mean matrix of the process $(Z_{2a,t}^{\eps,\diamond}, Z_{2d,t}^{\eps,\diamond})$. This eigenvalue is positive for all small enough $\eps>0$ and converges to $\widetilde \lambda$ as $\eps \downarrow 0$. Hence, there exists a function $f \colon (0,\infty) \to (0,\infty)$ with $\lim_{\eps \downarrow 0} f(\eps)=0$ such that for all $\eps>0$ sufficiently small,
\[ \Big| \frac{\widetilde\lambda^{(\eps,\diamond)}}{\widetilde\lambda}-1\Big| \leq \frac{f(\eps)}{2}. \numberthis\label{eigenvaluesclose}\]
Let us fix $\eps$ small enough such that \eqref{eigenvaluesclose} holds. Then from the second line of \eqref{upsilondef} we deduce that
\[ \P \Big( T^{(\eps,-),2}_{\eps^\xi} \leq T^{(\eps,-),2}_{0} \wedge \frac{\log K}{\widetilde \lambda} (1+f(\eps)),A_\eps \Big) \leq \P \Big( T^{2}_{\eps^\xi} \leq T^{2}_{0} \wedge \frac{\log K}{\widetilde \lambda} (1+f(\eps)),A_\eps \Big). \]
Using this together with the independence between $A_\eps$ and $(Z_{2a,t}^{\eps,\diamond}, Z_{2d,t}^{\eps,\diamond})$ and employing \cite[Section 7.5]{AN72}, we obtain for $\eps>0$ small enough (in particular such that $f(\eps)<1$)
\[
\begin{aligned}
& \liminf_{K \to \infty} \P \Big( T^{(\eps,-),2}_{\eps^\xi} \leq T^{(\eps,-),2}_{0} \wedge \frac{\log K}{\widetilde \lambda} (1+f(\eps)),A_\eps \Big) 
%\\
%\geq & \liminf_{K \to \infty} \P \Big( T^{(\eps,-),2}_{\eps^\xi} \leq \frac{\log K}{\widetilde \lambda} (1+f(\eps)) \Big) \P(A_\eps) \\
%\geq & \liminf_{K \to \infty} \P \Big( T^{(\eps,-),2}_{\eps^\xi} \leq \frac{\log K}{\widetilde \lambda^{(\eps,-)}} \big( 1-\frac{f(\eps)}{2} \big) (1+f(\eps))  \Big) \P(A_\eps) \\
%\geq &  \liminf_{K \to \infty} \P \Big( T^{(\eps,-),2}_{\eps^\xi} \leq \frac{\log K}{\widetilde \lambda^{(\eps,-)}} \Big) \P(A_\eps) \\
%\geq & 
\geq (1-q^{(\eps,-)})(1-o_\eps(1)). 
\end{aligned} 
\]
This inequality follows from computations that are analogous to \cite[Section 3.1.3, first display below (3.41)]{C+19}. %\color{red} Again, the homogamy paper is estimating $\limsup$s from below and $\liminf$s from above, which in my opinion is senseless. And I don't see why the (homogamy analogue of the) first inequality should be an equality; for us, an inequality is certainly sufficient at the moment. \color{black} \\
%\color{red} For the last inequality in \eqref{festimates} we need a proper reference from \cite{AN72}. The homogamy paper says ``using classical results on bi-type branching processes \cite{AN72}" and derives an analogous inequality without further comments. I haven't yet been able to figure out where this precisely is in the book. \color{black} 
Similarly, using the second line of \eqref{upsilondef}, we derive that for all sufficiently small $\eps>0$ 
%\begin{align*}
    %\P \Big( T^{(\eps,+),2}_{\eps^\xi} \leq T^{(\eps,+),2}_{0} \wedge \frac{\log K}{\widetilde \lambda} (1+f(\eps)),A_\eps \Big) \leq \P \Big( T^{2}_{\eps^\xi} \leq T^{2}_{0} \wedge \frac{\log K}{\widetilde \lambda} (1+f(\eps)),A_\eps \Big),
%\end{align*}
%and arguments similar to the ones used in \eqref{festimates} imply that
\[ \limsup_{K \to \infty} \P \Big( T^{(\eps,+),2}_{\eps^\xi} \leq T^{(\eps,+),2}_{0} \wedge \frac{\log K}{\widetilde \lambda} (1+f(\eps)),A_\eps \Big) \leq (1-q^{(\eps,+)})(1+o_\eps(1)). \]
These together imply \eqref{invasiontime}, hence the proof of the proposition is finished.
\end{proof}

\subsection{The second phase of invasion: Lotka--Volterra phase}\label{sec-secondphase}
\subsubsection{Convergence to a dynamical system for large population size.}
Now we rigorously state in what sense our population process $(\mathbf N_t^K)_{t \geq 0}$ is close to the solution $(\mathbf n_t)_{t \geq 0}=(n_1(t),n_{2a}(t),n_{2d}(t))_{t \geq 0}$ of the system of ODEs \eqref{3dimlotkavolterra} for large $K$ given that the corresponding initial conditions are close to each other.  As for \eqref{3dimlotkavolterra}, note that the vector field is locally Lipschitz and solutions do not explode in finite time, which guarantees existence and uniqueness for a given initial condition. Let $\mathbf n^0 =(n_1^0,n_{2a}^0,n_{2d}^0) \in [0,\infty)^3$ be an initial condition, and let $(\mathbf n^{(\mathbf n^0)}(t))_{t \geq 0}$ be the unique solution of the ODE started from the initial condition $\mathbf n^0$. Then, \cite[Theorem 2.1, p.~456]{EK} implies the following. 

\begin{lemma}\label{lemma-EKconvergence}
Let $T>0$. Assume that $(\mathbf N_0^K)_{K \geq 1}$ converge in probability to some deterministic vector $\mathbf n^0 = (n_1^0,n_{2a}^0,n_{2d}^0) \in [0,\infty)^3$ as $K$ tends to infinity. Then
\[ \lim_{K \to \infty} \sup_{0\leq s \leq T} \big| \mathbf N^K(s)-\mathbf n^{(\mathbf n^0)}(s)\big| =0 \]
in probability, where $\vert \cdot \vert$ denotes the Euclidean norm on $\R^3$. %\color{red} I've changed $\Vert \cdot \Vert_{\infty}$ to $\vert \cdot \vert$ because on $\R^3$ we don't need fancy norms. \color{black}
\end{lemma}
\subsubsection{Mutant active--dormant proportions.}
%\color{red} We need to start an ODE from an initial condition with at least one mutant coordinate being positive. Thus, the most senseful way of guarantee this is that the active and the dormant population gets into equilibrium, and thus possibly we will need to consider also proportions of actives and dormants ---sadly. \color{black} \\
On the event $\{ T_{\sqrt \eps}^2 < T_0^2 \wedge R_{2\eps} \} \subset A_{\eps}$, after time $T_{\eps}^2$ the total mutant population has size close to $\eps K$. Note that Proposition~\ref{prop-firstphase} provides us no coordinatewise information about the mutant population at this point in time. However, in order to guarantee convergence of the rescaled population process $(\mathbf N_t^K)_{t \geq 0}$ to a corresponding solution of the system of ODEs \eqref{3dimlotkavolterra}, we have to guarantee convergence of the initial conditions. We will thus show that with high probability, there exists a point in time in the interval $[T_\eps^2,T_{\sqrt \eps}^2]$ such that at this time, the resident population is still close to equilibrium, the total mutant population size is still at least of order $\eps K$ and the proportion of active and dormant mutants is close to the equilibrium proportion of the approximating branching process $((\widehat Z_{2a}(t),\widehat Z_{2d}(t)))_{t \geq 0}$. The present section is devoted to this problem. Next, in Section~\ref{sec-ODEconvergence}, we show that the ODE system \eqref{3dimlotkavolterra} started from the limiting initial condition converges to $(0,\bar n_{2a},\bar n_{2d})$ as $t \to \infty$.
%%%%%%%%%explanation why we maybe don't need this START
%\color{red} I'm wondering if this is really true. Fix $\eps>0$ small enough. Observe that we have
% \[ \lim_{K \to \infty} \frac{N^K_{2a,T^2_\eps}+N^K_{2d,T^2_\eps}}{K}= \eps. \]
%Since $K \mapsto (N^K_{2a,T^2_\eps},N^K_{2d,T^2_\eps})$ is bounded in $\R^2$, there exists a subsequence converging to some $(x_{2a,\eps},x_{2d,\eps})$. Then both coordinates of this subsequential limit are positive and their sum is equal to $\eps$. I hope that it is enough to handle the solutions of the dynamical system starting from any $(x_{2a,\eps},x_{2d,\eps})$ satisfying these conditions. \\
%So if we can prove that from any initial condition that is nonnegative and \emph{at least one} dormant coordinate is nonnegative the dynamical system \eqref{3dimlotkavolterra} converges to its unique stable equilibrium, then we don't need to handle active/dormant proportions at all in order to verify our theorem ---at least along a subsequence. Below the text continues a bit as if it were necessary to handle proportions: \color{black}
%%%%%%%%explanation why we maybe don't need this END

Since $\widetilde \lambda$ is positive, the Kesten--Stigum theorem (see e.g.~\cite[Theorem 2.1]{GB03}) ensures that
we have
\[ \Big( \frac{\widehat Z_{2a}(t)}{\widehat Z_{2a}(t)+\widehat Z_{2d}(t)},\frac{\widehat Z_{2d}(t)}{\widehat Z_{2a}(t)+\widehat Z_{2d}(t)}\Big) \underset{K \to \infty}{\longrightarrow} (\pi_{2a},\pi_{2d}) \]
on the event of survival of the approximating branching process $((\widehat Z_{2a,t},\widehat Z_{2d,t}))_{t \geq 0}$,
where $(\pi_{2a},\pi_{2d})$ is the positive left eigenvector of $J$ defined in \eqref{Jdef} associated to $\widetilde \lambda$ such that $\pi_{2a}+\pi_{2d}=1$, which can be computed explicitly according to \eqref{lambdatildedef}.  We verify the next proposition, employing some arguments of \cite[Proposition 3.2]{C+19}.
\begin{prop}\label{prop-secondphase}
There exists $C>0$ sufficiently large such that for $\delta>0$ such that $\pi_{2a} \pm \delta \in (0,1)$, under the same assumptions as Proposition~\ref{prop-firstphase},
\[ \numberthis\label{mutantproportions}
\begin{aligned}
& \liminf_{K \to \infty} \P \Big( \exists t \in \big[ T_\eps^2, T^2_{\sqrt \eps}\big], \frac{\eps K}{C} \leq K N_{2,t}^K \leq \sqrt \eps K, \\ & \qquad \pi_{2a}-\delta < \frac{ N_{2a,t}^K}{ N_{2a,t}^K+N_{2d,t}^K} < \pi_{2a}+\delta \Big| T^2_{\sqrt \eps} < T_0^2 \wedge R_{2\eps} \Big) \geq 1-o_\eps(1).
\end{aligned}
\]
\end{prop}
%\color{red} I have started writing down a proof for this proposition, analogously to the homogamy paper (decomposition of our Markov chain into a martingale part and a finite variation part using Poisson measures), however, I arrived at a point where things get technical and not analogous to our situation (because we have switches and no mutant births). So I put the proof fragment into a comment. 
%\color{red} Hopefully we don't need the proposition, see above. If we do, it's still a question how to prove it with so little monotonicity in the system. \color{black}

\begin{proof}
If $\pi_{2a}-\delta <\frac{N_{2a,T_\eps^2}^K}{N_{2a,T_\eps^2}^K+N_{2d,T_\eps^2}^K}  < \pi_{2a}+\delta$, then there is nothing to show. Let us assume that
\[ \frac{N_{2a,T_\eps^2}^K}{N_{2a,T_\eps^2}^K+N_{2d,T_\eps^2}^K} \leq \pi_{2a}-\delta, \]
the symmetric case $\frac{N_{2a,T_\eps^2}}{N_{2a,T_\eps^2}+N_{2d,T_\eps^2}} \geq \pi_{2a}+\delta$ can be treated similarly. Let us introduce the event
\[ \widetilde A_\eps := \{ T^2_{\sqrt \eps} < T_0^2 \wedge R_{2\eps} \} \]
on which we conditioned in \eqref{mutantproportions}. Our first goal is to show that for $\eps>0$ small, with high probability, once the total mutant population size reaches $\eps K$, for sufficiently large $C>0$ it will not decrease to a level lower than $\eps K/C$ again before it reaches $\sqrt \eps K$. To be more precise, for $C>0$ we introduce the stopping time
\[ T_{\eps,\eps/C} = \inf \big\{ t \geq T_\eps^2 \colon N_{2,t}^K \leq \smfrac{\eps K}{C} \big\}. \]
Then our goal is to show that if $C$ is large enough, then $T^2_{\sqrt \eps}$ is larger than $T_\eps^2 + \log\log(1/\eps)$ and smaller than $T_{\eps,\eps/C}$. First of all, for all $\eps>0$ sufficiently small, since the coupling \eqref{upsilondef} is satisfied on $\widetilde A_\eps$ and the branching processes $(Z_{2a,t}^{\eps,-},Z_{2d,t}^{\eps,-})$ is supercritical, \cite[Lemma A.1]{C+19} implies that for $C$ large enough,
\[ \lim_{K \to \infty} \P \big( T_{\eps,\eps/C} < T^2_{\sqrt \eps} \big| \widetilde A_\eps \big) = 0. \numberthis\label{thisalso} \]
On the other hand, note that the total size of mutant individuals is stochastically dominated from above by a Yule process with birth rate $\lambda_2$. Thus, by \cite[Lemma A.2]{C+19}, we have
\[ \lim_{K \to \infty} \P \big( T^2_{\sqrt \eps} \leq T^2_\eps + \log \log (1/\eps) \big| \widetilde A_\eps \big) \leq \sqrt \eps (\log (1/\eps))^{\lambda_2}. \numberthis\label{Yulebound} \]
Using these, we want to show that the fraction $\frac{N_{2a,t}^K}{N_{2a,t}^K+N_{2d,t}^K} $ cannot stay below $\pi_{2a}-\delta$ on $[T_{\eps}^2,T_{\sqrt \eps}^2]$ with probability close to one. Let us define the following five independent Poisson random measures on $[0,\infty]^2$ with intensity $\d s \d \theta$:
\begin{itemize}
    \item $P_{2a}^{\rm b} (\d s, \d \theta)$ representing the birth events of the active mutant individuals,
    \item $P_{2a}^{\rm d} (\d s, \d \theta)$ representing the death events of the active mutant individuals,
    \item $P_{2a\to 2d}^{\rm s} (\d s, \d \theta)$ representing the active$\to$dormant switching events,
    \item $P_{2d}^{\rm d} (\d s, \d \theta)$ representing the death events of the dormant mutant individuals (for $\kappa=0$ this measure can be omitted),
    \item $P_{2d\to 2a}^{\rm s}(\d s \d \theta)$ representing the dormant$\to$active switching events.
\end{itemize}
The reason why competitive death events can be assumed as independent of active$\to$dormant switches is that the corresponding Poisson random measures can be obtained as an independent thinning of a Poisson random measure with survival probability $1-p$ respectively the complementary thinning (with survival probability $p$), which are independent Poisson random measures according to \cite[Section 5.1]{K93}. 
%Equivalently but more closely to our model definition from Section~\ref{sec-modeldef}, one could introduce $P_{2a}^{\rm d} (\d s, \d \theta)$ and 
%$P_{2a\to 2d}^{\rm s} (\d s, \d \theta)$ as follows. Let $P_{2a}^{\rm nat.~d}(\d s, \d \theta)$ and $P_{2a}^{\rm comp.} (\d s, \d \theta)$ be two i.i.d.~Poisson random measures on $[0,\infty]^2$ with intensity $\d s \d \theta$, also independent of the Poisson measures for birth, dormant death and dormant$\to$active resuscitation. These two additional measures represent the natural death events of the active mutants and the competitive events affecting active mutants, respectively. Then one defines the active$\to$dormant switching measure $P_{2a}^{\rm d} (\d s, \d \theta)$ as the superposition of $P_{2a}^{\rm d} (\d s, \d \theta)$ and an independent thinning of $P_{2a}^{\rm comp.} (\d s, \d \theta)$ with success probability $p$, $P_{2a\to 2d}^{\rm s} (\d s, \d \theta)=P_{2a\to 2d}^{\rm s} (\d s, \d \theta)-P_{2a}^{\rm comp.} (\d s, \d \theta)$. Then, $P_{2a\to 2d}^{\rm s} (\d s, \d \theta)$
%
%\color{red} Do we want to introduce $P^{\rm c}_{2}(\d s, \d \theta)$ as the PPP for competitive events and define $P_{2a\to 2d}^{\rm s} (\d s, \d \theta)$ and the competitive death part of $P_{2a}^{\rm d} (\d s, \d \theta)$ as independent thinnings of this? \color{black} \color{blue}Perhaps mention this briefly, it is elegant.\color{black}\\ 
Let 
\[ \widetilde P_{2a}^{\rm b}(\d s, \d \theta):=P_{2a}^{\rm b}(\d s, \d \theta)-\d s \d \theta,\ldots,\widetilde P_{2d\to 2a}^{\rm s}(\d s,\d \theta):=P_{2d\to 2a}^{\rm s}(\d s,\d \theta)-\d s \d \theta \]
be the associated compensated measures. 
\begin{comment}{
Note that the mutant population sizes can be written as
\[
\begin{aligned}
N_{2a,t}^K & = N_{2a,0}^K + \int_0^t \int_{[0,\infty)} \Big( \mathds 1 \{ \theta \leq \lambda_2 N_{2a,s-}^K \} P_{2a}^{\rm b}(\d s, \d \theta) \\ & \quad - \mathds 1 \{ \theta \leq N_{2a,s-}^K(\mu+(1-p)\smfrac{\alpha}{K} (N_{1,s-}^K+N_{2a,s-}^K)) P_{2a}^{\rm d}(\d s, \d \theta) \\ & \qquad - \mathds 1 \{ \theta \leq p\alpha N_{2a,s-}^K  (N_{1,s-}^K+N_{2a,s-}^K)) \} P_{2a \to 2d}^{\rm s}(\d s, \d \theta)  \Big), \\ 
N_{2d,t}^K & = N_{2d,0}^K + \int_0^t \int_{[0,\infty)} \Big( \mathds 1 \{ \theta \leq p\alpha N_{2a,s-}^K  (N_{1,s-}^K+N_{2a,s-}^K))\} P_{2a \to 2d}^{\rm s}(\d s, \d \theta) 
\\ & \quad - \mathds 1 \{ \theta \leq \kappa\mu N_{2d,s-}^K \} P_{2d}^{\rm d}(\d s, \d \theta) - \mathds 1 \{ \theta \leq N_{2d,s-}^K \sigma \} P_{2d\to 2a}^{\rm s} \d s,\d \theta) \Big).
\end{aligned}
\]}\end{comment}
The fraction $\frac{N_{2a,t}^K}{N_{2a,t}^K+N_{2d,t}^K} $ is a semimartingale and can be decomposed as follows
\[ \frac{N_{2a,t}^K}{N_{2a,t}^K+N_{2d,t}^K}=\frac{N_{2a,T_\eps^2}^K}{N_{2a,T_\eps^2}^K+N_{2d,T_\eps^2}^K} + M_2(t)+ V_2(t), \qquad  t \geq T^2_\eps, \]
with $M_2$ being a martingale and $V_2$ a finite variation process such that
\begin{align*}
    M_2(t) & = \int_{T^2_\eps}^t \int_{[0,\infty)} \mathds 1 \{ \theta \leq \lambda_2 N_{2a,s-}^K \} \frac{N_{2d,s-}^K}{N_{2,s-}^K(N_{2,s-}^K+1)}  \widetilde P_{2a}^{\rm b} (\d s, \d \theta) \\
    & \quad - \int_{T^2_\eps}^t \int_{[0,\infty)} \mathds 1 \{ \theta \leq N_{2a,s-}^K(\mu+\alpha (1-p) (N_{1,s-}^K+N_{2a,s-}^K)) \} \frac{N_{2d,s-}^K}{N_{2,s-}^K(N_{2,s-}^K-1)}\widetilde P_{2a}^{\rm d} (\d s, \d \theta) \\
    & \quad - \int_{T^2_\eps}^t \int_{[0,\infty)} \mathds 1 \{ \theta \leq N_{2a,s-}^K(\alpha p (N_{1,s-}^K+N_{2a,s-}^K)) \} \frac{1}{N_{2,s-}^K}\widetilde P_{2a \to 2d}^{\rm s} (\d s, \d \theta) \\
     & \quad + \int_{T^2_\eps}^t \int_{[0,\infty)} \mathds 1 \{ \theta \leq \kappa \mu N_{2d,s-}^K \} \frac{N_{2a,s-}^K}{N_{2,s-}^K(N_{2,s-}^K-1)}\widetilde P_{2d}^{\rm d} (\d s, \d \theta) \\
     & \quad +  \int_{T^2_\eps}^t \int_{[0,\infty)} \mathds 1 \{ \theta \leq \sigma N_{2d,s-}^K \} \frac{1}{N_{2,s-}^K} \widetilde P_{2d \to 2a}^{\rm s} (\d s, \d \theta)
\end{align*}
and
\begin{align*}
    V_2(t) & = \int_{T^2_\eps}^t \Big\{ \lambda_2 N_{2a,s}^K \frac{N_{2d,s}^K}{N_{2,s}^K(N_{2,s}^K+1)}- N_{2a,s}^K (\mu+\alpha (1-p) (N_{1,s}^K+N_{2a,s}^K))  \frac{N_{2d,s}^K}{N_{2,s}^K(N_{2,s}^K-1)} \\ & \quad - N_{2a,s}^K \alpha p  (N_{1,s}^K+N_{2a,s}^K)) \frac{1}{N_{2,s}^K } \d s + \kappa \mu N^K_{2d,s}   \frac{N_{2a,s}^K}{N_{2,s}^K(N_{2,s}^K-1)} + \sigma N^K_{2d,s}  \frac{1}{N_{2,s}^K} \Big\} \d s.
\end{align*}
%Hence, \color{red} do we need this? \color{black} there exists a finite constant $C>0$ such that
%\[ \sup_{t \in [T^2_\eps,T^2_{\sqrt \eps}]} V_2(t) \leq C \int_{T^2_\eps}^t \frac{N_{2a,t}^K}{N_{2,t}^K} \d s.  \]
Further, the predictable quadratic variation of the martingale $M_2$ is given as follows
\[
\begin{aligned}
\langle M_2 \rangle_t& = \int_{T^2_\eps}^t \lambda_2 N_{2a,s}^K \frac{(N_{2d,s}^K)^2}{(N_{2,s}^K)^2(N_{2,s}^K+1)^2} \d s \\ & \quad + \int_{T_\eps^2}^t \mu N_{2a,s}^K (\mu+\alpha (1-p) (N_{1,s-}^K+N_{2a,s-}^K))  \frac{(N_{2d,s}^K)^2}{(N_{2,s}^K)^2(N_{2,s}^K-1)^2} \\ & \quad + N_{2a,s}^K \alpha p  (N_{1,s-}^K+N_{2a,s-}^K) \frac{1}{(N_{2,s}^K)^2 } \d s + \kappa \mu N^K_{2d,s}   \frac{(N_{2a,s}^K)^2}{(N_{2,s}^K(N_{2,s}^K-1))^2} + \sigma N^K_{2d,s}  \frac{1}{(N_{2,s}^K)^2}  \d s.
\end{aligned}
\]
This yields that there exists $C_0>0$ such that for all $t \geq T_\eps^2$,
\[ \langle M_2 \rangle_t \leq C_0(t-T_\eps^2) \sup_{T_\eps^2 \leq s \leq t} \frac{1}{N_{2,s}^K-1}. \]
This implies
\[ \langle M_2 \rangle_{(T_\eps^2+\log \log (1/\eps)) \wedge T_{\eps,\eps/C}} \leq \frac{C_0 \log \log (1/\eps)}{\smfrac{\eps K}{C}-1}  \numberthis\label{QVbound} \]
and
\[ 
\begin{aligned}
 & \limsup_{K \to \infty}  \mathbb P \Big( \sup_{T_\eps^2 \leq t \leq T_\eps^2 + \log \log (1/\eps)} |M_2(t)| \geq \eps \Big| \widetilde A_\eps \Big) \\
& \leq \limsup_{K \to \infty} \Big( \mathbb P\Big( \sup_{T_\eps^2 \leq t \leq (T_\eps^2 + \log \log (1/\eps)) \wedge T_{\eps,\eps/C}} |M_2(t)| \geq \eps \Big| \widetilde A_\eps \Big)
+ \P \big( T_{\eps,\eps/C} < T_\eps^2 + \log \log (1/\eps) \big| \widetilde A_\eps \big) \Big)
\\ & \leq \limsup_{K \to \infty} \frac{1}{\eps^2} \E \Big[ \langle M_2 \rangle_{(T_\eps^2+\log \log (1/\eps)) \wedge T_{\eps,\eps/C}} \Big| \widetilde A_\eps \Big]  + \sqrt \eps (\log 1/\eps)^{\lambda_2}= \sqrt \eps (\log 1/\eps)^{\lambda_2}, \end{aligned} \numberthis\label{martingaleestimate} \] 
where in the first inequality of the last line we used Doob's martingale inequality for the first term and \eqref{thisalso} together with \eqref{Yulebound} for the second term, and the last inequality of the last line is due to \eqref{QVbound}. 

Let us now consider the finite variation process $V_2$. This can be written as
\[ V_2(t)=\int_{T_\eps^2}^t P \Big( \frac{N_{2a,s}^K}{N_{2,s}^K}\Big) \frac{N_{2,s}^K}{N_{2,s}^K + 1} + Q^{(s)} \Big( \frac{N_{2a,s}^K}{N_{2,s}^K}\Big) \frac{N_{2,s}^K}{N_{2,s}^K - 1} + R^{(s)} \Big( \frac{N_{2a,s}^K}{N_{2,s}^K}\Big) \d s, \numberthis\label{V2transcript} \]
with
\begin{multline*} P (x) = \lambda_2 x(1-x), \quad Q^{(s)}(x)=(\kappa\mu-\mu-\alpha (1-p)(N_{1,s}^K+N_{2a,s}^K))) x(1-x), \\  R^{(s)}(x) =\sigma (1-x)-p\alpha\big(N_{1,s}^K+N_{2a,s}^K\big)x.   \end{multline*}
For $\eps>0$ small, on $[T_\eps^2,T_{\sqrt \eps}^2]$,  $Q^{(s)}$ and $R^{(s)}$ are close on $[0,1]$, respectively, to the polynomial functions $Q,R$ given as follows
\begin{multline*} \quad Q(x)=(\kappa\mu-\mu-\alpha(1-p)\bar n_1) x(1-x)=(\kappa\mu-\mu-(1-p)(\lambda_1-\mu))x(1-x), \\  R(x) =\sigma (1-x)-p\alpha\bar n_1 x=\sigma(1-x)-p(\lambda_1-\mu)x.   \end{multline*}
Thus, for given $\eps>0$, for all sufficiently large $K$, the integrand in \eqref{V2transcript}  is close to the polynomial function
\[ S(x)=(\lambda_2+\kappa\mu-\mu-(1-p)(\lambda_1-\mu)) x(1-x)+\sigma(1-x)-p(\lambda_1-\mu)x.\]
Since $S(0)>0$ and $S(1) < 0$, further, $S$ is of degree 2, the equation $\dot x=S(x)$ has a unique equilibrium in $(0,1)$. Now, let $(\pi_{2a},\pi_{2d})$ be the left eigenvector of the matrix $J$ defined in \eqref{Jdef} corresponding to the eigenvalue $\widetilde \lambda$ such that $\pi_{2a}+\pi_{2d}=1$. A direct computation implies that $\pi_{2a}$ is a root of $S$ and thus equal to this equilibrium. Thus, we can choose $\delta>0$ and $\theta>0$ such that $\pi_{2a}-\delta >0$ and for all $x < \pi_{2a}-\delta$, $S(x)>\theta/2$. By continuity, this implies that for all sufficiently small $\eps>0$ and accordingly chosen sufficiently large $K>0$, on the event $\widetilde A_\eps$ the following relation holds
\begin{equation}\label{polynomialcontinuity} 
    \forall s \in \big[ T_\eps^2, T^2_{\sqrt\eps} \big], \forall x \in (0,\pi_{2a}-\delta),
    P(x) \frac{N_{2,s}^K+1}{N_{2,s}^K}+ Q^{(s)} (x) \frac{N_{2,s}^K-1}{N_{2,s}^K} +R^{(s)}(x) \geq \frac{\theta}{2}>0. 
\end{equation}
Let us define
\[ \mathfrak t_{2a}^{(\eps)} : = \inf \Big\{ t \geq T_\eps^2 \colon \frac{N_{2a,t}^K}{N_{2d,t}^K} \geq \pi_{2a}-\delta \Big\}. \]
From \eqref{martingaleestimate} and \eqref{polynomialcontinuity} we obtain that on the event $\widetilde A_\eps$, for any $t \in [T_\eps^2,(T_\eps^2 + \log \log(1/\eps)) \wedge \mathfrak t_{2a}^{(\eps)} ]$,
\[ \pi_{2a}-\delta \geq \frac{N_{2a,t}^K}{N_{2,t}^K} \geq \frac{\theta}{2} \Big( \log\log(1/\eps) \wedge (\mathfrak t_{2a}^{(\eps)}-T_\eps^2) \Big) - \eps  \]
with a probability higher than $1-\sqrt \eps (\log(1/\eps))^{\lambda_2}$. Since $\smfrac{\theta}{2} \log \log (1/\eps)$ tends to $\infty$ as $\eps \downarrow 0$, it follows that for $\eps>0$ small, $\mathfrak t_{2a}^{(\eps)}$ is smaller than $T_\eps^2 + \log \log (1/\eps)$ and thus smaller than $T_{\sqrt \eps}^2$ with a probability close to 1 on the event $\widetilde A_\eps$, where we also used \eqref{Yulebound}.

Finally, note that each jump of the process $N_{2a,t}^K/N_{2,t}^K$ is smaller than $(\eps K/C+1)^{-1}$, and hence smaller than $\delta$ for all $K$ sufficiently large (given $\eps$). Thus, after the time $\mathfrak t_{2a}^{(\eps)}$, the process will be contained in the interval $[\pi_{2a}-\delta,\pi_{2a}+\delta]$ for some positive amount of time. Hence, we conclude the proposition.
\end{proof}
\subsubsection{Convergence of the dynamical system for large times.}\label{sec-ODEconvergence}
In this section, we first investigate the stability of the equilibria of the system of ODEs \eqref{3dimlotkavolterra} via linearization. Then we show convergence of the solution of the system to $(0,\bar n_{2a},\bar n_{2d})$ for initial conditions corresponding to Proposition~\ref{prop-secondphase}, and for the two-dimensional projection of the system even for any nonnegative initial condition that has at least one nonzero coordinate. As mentioned before, the behaviour of the dynamical system is rather different from the ones described in \cite{C+16, C+19}. 
\begin{prop}\label{prop-stableequilibrium3D}
Assume that \eqref{invasionpossible} holds. Then the system of ODEs \eqref{3dimlotkavolterra} admits precisely three equilibria: $(0,0,0)$, $(\bar n_1,0,0)$ and $(0,\bar n_{2a},\bar n_{2d})$, the first two of which are unstable, whereas the third one is asymptotically stable.
\end{prop}
\begin{proof}
We easily identify the equilibria $(0,0,0)$,$(\bar n_1,0,0)$ and $(0,\bar x_a,\bar x_d)$, and we claim that further equilibria do not exist. Indeed, it is easy to see that apart from $(0,0,0)$, the only possible coordinatewise nonnegative equilibrium of the form $(0,\cdot,\cdot)$ is $(0,\bar x_a,\bar x_d)$ and the only possible one of the form $(\cdot,0,\cdot)$ or $(\cdot,\cdot,0)$ is $(\bar n_1,0,0)$. Hence, it remains to exclude the existence of equilibria with three positive coordinates. For such equilibria $(n_1,n_{2a},n_{2d})$, expressing $n_1$ from the first line of \eqref{3dimlotkavolterra} and substituting it into the second and third line divided by $n_{2a}$ yields
    \[ \frac{n_{2d}}{n_{2a}}=\frac{\lambda_1-\lambda_2}{\sigma}=\frac{1}{\kappa\mu+\sigma} p (\lambda_1-\mu), \]
    but the last inequality contradicts with \eqref{invasionpossible}. We conclude the claim. \\
    We continue with checking stability of the three equilibria. At any equilibrium $(n_1,n_{2a},n_{2d})$, the Jacobian matrix is given as
    \[ B(n_1,n_{2a},n_{2d}) = \begin{small}\begin{pmatrix}
    \lambda_1-\mu-2\alpha n_1-\alpha n_{2a} & -\alpha n_1 & 0 \\
    -\alpha n_{2a} & \lambda_2-\mu-2\alpha n_{2a}-\alpha n_1 & \sigma \\
    p\alpha n_{2a} & 2p\alpha n_{2a} + p \alpha n_{1} & -(\kappa\mu+\sigma)
    \end{pmatrix}\end{small}.
    \] As for the origin, $B$ takes the block diagonal form
    \[ B(0,0,0) = \begin{pmatrix}
    \lambda_1-\mu & 0 & 0 \\
   0 & \lambda_2-\mu & \sigma \\
    0 & 0 & -(\kappa\mu+\sigma)
    \end{pmatrix}.
    \]
    Its spectrum is the union of the spectra of the two blocks, hence $\lambda_1-\mu$ is an eigenvalue (with eigenvector $(1,0,0)$). Since this eigenvalue is positive, the origin is unstable. At $(\bar n_1,0,0)$, since $\alpha n_1=\lambda_1-\mu$, the Jacobian matrix is
    \[ B(\bar n_1,0,0)=\begin{pmatrix}
    -\lambda_1+\mu & -\lambda_1+\mu & 0 \\
    0 & \lambda_2-\lambda_1 & \sigma \\
   0 & p(\lambda_1-\mu) & -(\kappa\mu+\sigma)
    \end{pmatrix}. \numberthis\label{detdivision}
    \]
    The determinant of this matrix is
    \[ \det B(\bar n_1,0,0)=-(\lambda_1-\mu)((\lambda_2-\lambda_1)(-\kappa\mu-\sigma)-p(\lambda_1-\mu)\sigma). \]
    Now, since $\lambda_1>\mu$,
    further, thanks to \eqref{invasionpossible},
    \[ (\lambda_1-\lambda_2)(\kappa\mu+\sigma)-p(\lambda_1-\mu)\sigma < 0, \]
    the determinant is positive. Hence, in order to conclude that the equilibrium is unstable, it suffices to show that all eigenvalues are real. This follows from the fact that by \eqref{detdivision}, $\det  B(\bar n_1,0,0)/(\mu-\lambda)$ is negative. Since this quotient equals the product of the two other eigenvalues of the matrix, it is impossible that these eigenvalues are complex (and thus conjugate).
    % Now, we observe that $(1,0,0)$ is an eigenvector of $B(\bar n_1,0,0)$ corresponding to the eigenvalue $\lambda_1-\mu-2\alpha\bar n_1= \mu-\lambda_1<0$. Dividing the characteristic equation of $\det B(\bar n_1,0,0)$ (in variable $\lambda$) by $\lambda_1-\mu-2\alpha\bar n_1 -\lambda$ and using that $\alpha \bar n_1=\lambda_1-\mu$, we obtain
    %\[ \lambda^2 + (\lambda_1-\lambda_2+\kappa \mu+\sigma)\lambda -(\lambda_2-\lambda_1)(\kappa\mu+\sigma)-\sigma p (\lambda_1-\mu)=0. \]
    %The discriminant of this quadratic equation is equal to 
   % \[ (\lambda_1-\lambda_2+\kappa\mu+\sigma)^2-4(\lambda_1-\lambda_2)(\kappa\mu+\sigma)=(\lambda_1-\lambda_2-\kappa\mu-\sigma)^2 \geq 0,\]
   % hence all eigenvalues are real. 
   Finally, let us consider the equilibrium $(0,\bar n_{2a},\bar n_{2d})$. We have 
     \[ B(0,\bar n_{2a},\bar n_{2d}) = \begin{pmatrix}
    \lambda_1-\mu-\alpha \bar n_{2a} & 0 & 0 \\
    0 & \lambda_2-\mu-2\alpha \bar n_{2a} & \sigma \\
    p\alpha \bar n_{2a} & 2p\alpha \bar n_{2a} & -(\kappa\mu+\sigma)
    \end{pmatrix}.
    \]
   % Let us show that
    %\[ \det B  (0,\bar n_{2a},\bar n_{2d}) =(\lambda_1-\mu-\alpha \bar n_{2a})\big[ (\lambda_2-\mu-2\alpha \bar n_{2a})(-\kappa\mu-\sigma)-2\sigma p\alpha \bar n_{2a} \big]<0. \numberthis\label{xaxddetnegative} \]
    We have already seen in Section~\ref{sec-invasionconditions} that $\lambda_1-\mu-\alpha \bar n_{2a} <0$ under condition \eqref{invasionpossible}, and this quantity is clearly an eigenvalue of the matrix $ B  (0,\bar n_{2a},\bar n_{2d})$. The other two ones are the eigenvalues of the matrix $A(\bar n_{2a},\bar n_{2d})$ (cf.~\eqref{Jacobian}), which are negative since $\lambda_2>\mu$, see also Section~\ref{sec-invasionconditions}. We conclude that $B(0,\bar n_{2a},\bar n_{2d})$ is negative definite and hence the equilibrium $(0,\bar n_{2a},\bar n_{2d})$ is asymptotically stable under condition \eqref{invasionpossible}.
    %On the other hand, 
    %\begin{align*}
     %    (\lambda_2-\mu-2\alpha \bar n_{2a})(-\kappa\mu-\sigma)-2\sigma p\alpha \bar n_{2a} & = (\lambda_2-\mu-\alpha \bar n_{2a})(-\kappa\mu-\sigma) -\alpha \bar n_{2a}(-\kappa \mu-\sigma+p\sigma)  \\ & = \alpha \bar n_{2a}(\kappa \mu+\sigma)>0,\numberthis\label{usexadef}
   % \end{align*} 
   % where in the first inequality in \eqref{usexadef} we used the second equality of \eqref{positiveequilibrium}. This implies \eqref{xaxddetnegative}. Now, we observe that any vector of the form $(0,x,y)^T$, $x,y \in \R$, satisfies $ B(0,\bar n_{2a},\bar n_{2d}) (0,x,y)^T = (0,x',y')^T$ for some $x',y' \in \R$. Hence, in particular, $(0,v_{a,1},v_{d,1})$ and $(0,v_{a,2},v_{d,2})$ are eigenvectors of $ B(0,\bar n_{2a},\bar n_{2d}) $ where $(v_{a,1},v_{d,1})$ and $(v_{a,2},v_{d,2})$ are eigenvectors of $A(\bar n_{2a},\bar n_{2d})$ (cf.~\eqref{Jacobian}), corresponding to the same eigenvalues as for the matrix $A(\bar n_{2a},\bar n_{2d})$. We already know that these eigenvalues are negative, and hence \eqref{xaxddetnegative} implies that the third eigenvalue of $ B(0,\bar n_{2a},\bar n_{2d}) $ is also negative. Therefore, $(0,\bar n_{2a},\bar n_{2d})$ is asymptotically stable.
    \end{proof}
    %Knowing that the only stable equilibrium is $(0,\bar n_{2a},\bar n_{2d})$, our goal is now to show that for any nonnegative initial condition apart from the ones of the form $(n_1,0,0)$ with $n_1 \geq 0$, the solution of \eqref{3dimlotkavolterra} converges as $t \to \infty$ to this equilibrium. Let us recall that $(0,0,0)$ and $(\bar n_1,0,0)$ are unstable equilibria, and it is easy to see (from a direct inspection of the sign of $n_1(t)$) that started from $(n_1,0,0)$ with $n_1 >0$, the solution converges to $(\bar n_1,0,0)$. In general, the statement of convergence to $(0,\bar n_{2a},\bar n_{2d})$ for initial conditions where at least one mutant coordinate is strictly positive is not trivial, since a three-dimensional dynamical system could also have periodic or chaotic solutions, and the system \eqref{3dimlotkavolterra} also does not feature strong monotonicity properties.
    Now, for the two-dimensional variant 
    \begin{align*}
    \dot n_{2a}(t)&=n_{2a}(t)(\lambda_2-\mu-\alpha n_{2a}(t))+\sigma n_{2d}(t), \\
    \dot n_{2d}(t)&=p\alpha n_{2a}^2(t)-(\kappa\mu+\sigma) n_{2d}(t),
\numberthis\label{2doncemore}
\end{align*}
    of the system, introduced in \eqref{linearized}, which corresponds to starting the system \eqref{3dimlotkavolterra} from $\{ 0\} \times [0,\infty)^2$ and ignoring the invariant first coordinate, $(\bar n_{2a},\bar n_{2d})$ turns out to be the limit of the solution started from any nonnegative initial condition apart from $(0,0)$. Let us recall that this system has an asymptotically stable equilibrium $(\bar n_{2a},\bar n_{2d})$ and an unstable one $(0,0)$ under the assumption that $\lambda_2>\mu$.
    \begin{lemma}\label{lemma-2dODE}
       In case $(n_{2a}(0),n_{2d}(0)) \in [0,\infty)^2 \setminus \{ (0,0) \}$, we have \[ \lim_{t \to \infty} (n_{2a}(t),n_{2d}(t)) = (\bar n_{2a},\bar n_{2d}). \] 
    \end{lemma}
    \begin{proof}
Observe that the active coordinate of the stable equilibrium,
    \[ \bar n_{2a} = \frac{(\lambda_2-\mu)(\kappa\mu+\sigma)}{\alpha(\kappa\mu+(1-p)\sigma)}>0 \]
    satisfies
    \[ \frac{\lambda_2-\mu}{\alpha} < \bar n_{2a} \leq \frac{\lambda_2-\mu}{(1-p)\alpha},  \numberthis\label{xainequality} \]
    where the second inequality is an equality if and only if $\kappa=0$. Further, the dormant coordinate $\bar n_{2d}$ is positive. Note further that the divergence of the system is given as
    \[ \lambda_2-\mu-2\alpha n_{2a}(t)-(\kappa\mu+\sigma). \]
    This is certainly negative if $n_{2a} \geq \smfrac{\lambda_2-\mu}{2\alpha}$, $n_{2d} \geq 0$, and at least one of the latter two inequalities is strict. In particular, the Bendixson criterion~\cite[Theorem 7.10]{DLA06} implies that there is no nontrivial periodic solution in the open and simply connected set
    \[ U= \big\{ (n_{2a},n_{2d}) \in \mathbb R^2 \colon n_{2a} > \smfrac{\lambda_2-\mu}{2\alpha}, n_{2d} > 0 \big\}. \]
    Since this is a two-dimensional system and all solutions of the system with coordinatewise nonnegative initial conditions are bounded, this implies that any solution starting from $U$ converges to the equilibrium $(\bar n_{2a},\bar n_{2d}) \in U$. It remains to show that any solution started from $[0,\infty)^2 \setminus (\{ (0,0) \} \cup U)$ will enter the open set $U$ after finite time.
    
Now, observe that if $n_{2a}(0)> 0$ and $ n_{2d}(0) \geq 0$, then $\dot n_{2a}$ is positive and bounded away from zero until $n_{2a}$ reaches $\smfrac{\lambda_2-\mu}{2\alpha}$, hence $n_{2a}$ will reach this level. If $n_{2d}(0)>0$ and $n_{2a}(0)=0$, then there exists $\delta>0$ such that $n_{2a}(\delta)>0$ and $n_{2d}(\delta)>0$, and hence $n_{2a}$ will also reach the level $\smfrac{\lambda_2-\mu}{2\alpha}$ in finite time.  Further, for $t>0$, if $n_{2a}(t)=\smfrac{\lambda_2-\mu}{2\alpha}$ and $n_{2d}(t)\geq 0$, then plugging in the first inequality of \eqref{xainequality} to the first equation of \eqref{2doncemore} implies that $\dot n_{2a}(t)>0$. This implies that if $n_{2d}(t)>0$, then \[ (n_{2a}(t+\eps),n_{2d}(t+\eps)) \in U, \qquad \forall \eps>0 \text{ sufficiently small}. \numberthis\label{Uentrance} \]
Else, $\dot n_{2a}(t)=0$ but $\dot n_{2d}(t)>0$, and hence the observations of the previous case imply that $\dot n_{2a}(t+\eps)>0$ for all sufficiently small $\eps>0$, thus \eqref{Uentrance} also holds.
\end{proof}
Finally, we show convergence of the original 3-dimensional system to $(0,\bar n_{2a},\bar n_{2d})$ as $t \to \infty$ for initial conditions corresponding to Proposition~\ref{prop-secondphase}. In other words, we verify some global attractor properties of this equilibrium, which are not as general as for the two-dimensional system but sufficient for the goals of the present paper.
\begin{lemma}\label{lemma-3dODE}
Let us consider the system of ODEs \eqref{3dimlotkavolterra}. If the initial condition $(n_1,n_{2a},n_{2d})=(n_1(0),n_{2a}(0),n_{2d}(0))$ satisfies
\[ \frac{p\alpha(n_1+n_{2a})}{\kappa\mu+\sigma} > \frac{n_{2d}}{n_{2a}} > \frac{\mu-\lambda_2+\alpha(n_1+n_{2a})}{\sigma},  \qquad n_1 \geq 0, n_{2a},n_{2d}>0, \numberthis\label{proportioncond} \]
then 
\[ \lim_{t \to \infty} (n_1(t),n_{2a}(t), n_{2d}(t))=(0,\bar n_{2a},\bar n_{2d}). \numberthis\label{goodlimit} \]
\end{lemma}
Note that in the two-dimensional case $n_1(0)=0$, Lemma~\ref{lemma-3dODE} is weaker than Lemma~\ref{lemma-2dODE}. We will use the stronger assertion (more precisely, an approximative version of it) when handling the third phase of invasion in Section~\ref{sec-thirdphase}, where perturbations of the system \eqref{2dimlotkavolterra} need to be treated. 
\begin{proof}
Let us assume that for some $t \geq 0$,  $(n_1(t),n_{2a}(t),n_{2d}(t))=(n_1,n_{2a},n_{2d})$. Then the first inequality in \eqref{proportioncond} is equivalent to the statement that $\dot n_{2d}(t)>0$ and the second one is equivalent to the statement that $\dot n_{2a}(t)>0$. Hence, as long as \eqref{proportioncond} holds, $t \mapsto n_{2a}(t)$ and $t \mapsto n_{2d}(t)$ are strictly increasing. 

Let us assume that condition \eqref{proportioncond} holds for $(n_1,n_{2a},n_{2d})=(n_1(0),n_{2a}(0),n_{2d}(0))$. We claim that then it also holds for all $t > 0$ with $(n_1,n_{2a},n_{2d})=(n_1(t),n_{2a}(t),n_{2d}(t))$, unless eventually $n_{2a}(t)=\bar n_{2a}$ and $n_{2d}(t)=\bar n_{2d}$. Indeed, let us assume that for some $t>0$, $(n_1(t),n_{2a}(t),n_{2d}(t))$ lies on the boundary of the set
\[ \{ (n_1,n_{2a},n_{2d}) \in [0,\infty) \times (0,\infty) \times (0,\infty) \colon (n_1,n_{2a},n_{2d})\text{ satisfies \eqref{proportioncond}} \} \numberthis\label{propcondset} \]
with $n_{2a},n_{2d} \geq 0$, in such a way that $(n_1(s),n_{2a}(s),n_{2d}(s))$ is contained in the set \eqref{propcondset} for all $0 \leq s < t$. Then $n_{2a},n_{2d} >0$ holds because $n_{2a}(0),n_{2d}>0$ by assumption, moreover, $s \mapsto n_{2a}(s)$ and $s \mapsto n_{2d}(s)$ are increasing on $[0,t)$. Hence, one of the following conditions holds:
    \begin{enumerate}[(i)]
    \item\label{first-bdry} $\dot n_{2d}(t)=0$, $\dot n_{2a}(t)>0$,
    \item\label{second-bdry} $\dot n_{2a}(t)=0$, $\dot n_{2d}(t)>0$,
    \item\label{third-bdry} $\dot n_{2a}(t)=\dot n_{2d}(t)=0$.
    \end{enumerate}
In case \eqref{first-bdry} we have
\[ \dot{\big( \smfrac{n_{2d}}{n_{2a}} \big)}(t) =\frac{-\dot n_{2a}(t) n_{2d}(t)}{n_{2a}(t)^2} < 0. \]
The case \eqref{second-bdry} yields
\[ \dot{\big( \smfrac{n_{2d}}{n_{2a}} \big)}(t) =\frac{\dot n_{2d}(t) n_{2a}(t)}{n_{2a}(t)^2} > 0. \]
In case \eqref{third-bdry} we have (thanks to the condition that $n_{2a},n_{2d}>0$) that $(n_{2a},n_{2d})=(\bar n_{2a},\bar n_{2d})$. We conclude that if $(n_1,n_{2a},n_{2d})=(n_1(0),n_{2a}(0),n_{2d}(0))$ satisfies \eqref{proportioncond}, then $t \mapsto (n_1(t),n_{2a}(t),n_{2d}(t))$ never enters the complement of the closure of the set \eqref{propcondset} apart from $(\bar n_{2a},\bar n_{2d})$, which implies the claim.

Now, given that condition \eqref{proportioncond} holds for $(n_1,n_{2a},n_{2d})=(n_1(0),n_{2a}(0),n_{2d}(0))$, $t \mapsto n_{2a}(t)$ and $t \mapsto n_{2d}(t)$ are nonnegative, bounded, increasing, and strictly increasing unless $(n_{2a}(t),n_{2d}(t))=(\bar n_{2a},\bar n_{2d})$ eventually, in which case both coordinates would immediately become constant. Further, $t \mapsto n_1(t)$ is also bounded and nonnegative. Hence, $(n_1(t),n_{2a}(t),n_{2d}(t))$ converges along a subsequence to $(n_1^*, \bar n_{2a},\bar n_{2d})$ for some $n_1^* \geq 0$. Now we argue that $n_1^*$ must be equal to zero. Indeed, taking limits of \eqref{proportioncond} implies that
\[ \frac{p\alpha(n_1^*+\bar n_{2a})}{\kappa\mu+\sigma} \geq \frac{\bar x_{d}}{\bar x_{a}} \geq \frac{\mu-\lambda_2+\alpha(n_1^*+\bar n_{2a})}{\sigma}. \numberthis\label{endcond} \]
Observe that \eqref{endcond} holds for $n_1^*=0$ thanks to~\eqref{positiveequilibrium}. Taking this into account, any subsequential limit has to satisfy
\[ \frac{p\alpha n_1^*}{\kappa\mu+\sigma} \geq \frac{\alpha n_1^*}{\sigma}. \]
Since by our assumptions, $\smfrac{p}{\kappa\mu+\sigma} < \smfrac{1}{\sigma}$, we conclude that $\bar n_1^*=0$. Hence, \eqref{goodlimit} follows.
\end{proof}
The last auxiliary result corresponding to the second phase of invasion states that the state of the population process reached thanks to Proposition~\ref{prop-secondphase} belongs to the domain of attraction of the stable equilibrium $(0,\bar n_{2a},\bar n_{2d})$. 
\begin{lemma}\label{lemma-goodstart}
Let $C$ be chosen according to Proposition~\ref{prop-secondphase}, further, $n_1,n_{2a},n_{2d}>0$ such that $n_1 \in (\bar n_1-2\eps,\bar n_1+2\eps), n_{2a}+n_{2d} \in (\eps/C,\sqrt\eps)$, and $\smfrac{n_{2d}}{n_{2a}} = \smfrac{\pi_{2d}}{\pi_{2a}}$. Then, if $\eps>0$ is sufficiently small, then $(n_1,n_{2a},n_{2d})$ satisfies \eqref{proportioncond}.
\end{lemma}
\begin{proof}
Since $(\pi_{2a},\pi_{2d})$ is a left eigenvector of $J$ corresponding to the eigenvalue $\widetilde\lambda$, we have
\begin{align*}
    (\lambda_2-\lambda_1)+\sigma \frac{\pi_{2d}}{\pi_{2a}}=\widetilde\lambda=
    p(\lambda_1-\mu)\frac{\pi_{2a}}{\pi_{2d}}-(\kappa\mu+\sigma).
\end{align*}
Hence, since $\widetilde\lambda>0$, given that $\eps>0$ is small enough, we obtain
\[ \frac{\pi_{2d}}{\pi_{2a}} = \frac{\widetilde\lambda-\lambda_2+\lambda_1}{\sigma}=\frac{\widetilde\lambda-\lambda_2+\mu+\alpha \Big(\frac{\lambda_1-\mu}{\alpha}\Big)}{\sigma} > \frac{-\lambda_2+\mu+\alpha\Big(\frac{\lambda_1-\mu}{\alpha}+3\sqrt\eps\Big)}{\sigma} \geq \frac{\mu-\lambda_2+\alpha(n_1+n_{2a})}{\sigma}\]
and
\[ \frac{\pi_{2d}}{\pi_{2a}} =\frac{p\alpha\Big(\frac{\lambda_1-\mu}{\alpha}\Big)}{\widetilde\lambda+\kappa\mu+\sigma} < \frac{p\alpha\Big(\frac{\lambda_1-\mu}{\alpha}-2\eps\Big)}{\kappa\mu+\sigma}  \leq \frac{p\alpha(n_1+n_{2a})}{\kappa\mu+\sigma},\]
as asserted.
\end{proof}

\subsection{The third phase of invasion: extinction of the resident population}\label{sec-thirdphase}
After the second phase, the rescaled process $\mathbf N_t^K$ is close to the state $(0,\bar n_{2a},\bar n_{2d})$. In particular, $N_{1,t}^K$ is at most $\eps K$ for some $\eps>0$ small. In this subsection, we estimate the time of the extinction of the resident population. We also need to check that the mutant population stays close to $(\bar n_{2a},\bar n_{2d})$ during its time. We recall the set $S_\beta$ \eqref{Sbetadef} and the time $T_{S_\beta}$ \eqref{TSbetadef}. We have the following proposition.
\begin{prop}\label{prop-thirdphase}
There exist $\eps_0,C_0>0$ such that for all $\eps \in (0,\eps_0)$, under condition~\eqref{invasionpossible} with $\lambda_1>\lambda_2>\mu$, if there exists $\eta \in (0,1/2)$ that satisfies
\[ \big| N^K_{2a}(0)-\bar n_{2a} \big| \leq \eps~\text{ and }~\big| N^K_{2d}(0) - \bar n_{2d} | \leq \eps ~\text{ and }~\eta \eps/2 \leq N^K_1(0) \leq \eps/2, \]
then
\begin{align*}
    & \forall\widetilde C>(\mu+\alpha\bar n_{2a}-\lambda_1)^{-1}+C_0\eps, \quad & \P(T_{S_\eps} \leq \widetilde C \log K) \underset{K \to \infty}{\longrightarrow} 1, \\
    & \forall 0 \leq\widetilde C < (\mu+\alpha\bar n_{2a}-\lambda_1)^{-1}-C_0\eps, \quad & \P(T_{S_\eps} \leq \widetilde C \log K) \underset{K \to \infty}{\longrightarrow} 0.
\end{align*}
\end{prop}
\begin{proof}
Our first step is to show that the rescaled population size vector $(N_{2a,t}^K,N_{2d,t}^K)$ stays close to its equilibrium $(\bar n_{2a},\bar n_{2d})$ for long times, given that the resident population is small. To this aim, we employ arguments  similar to the ones of \cite[Proof of Proposition 4.1, Step 1]{C+16}. For $\eps>0$ we define the stopping times
\[ R_{\eps,i} = \inf \big\{ t \geq 0 \colon \big| N_{2i,t}^K-\bar x_i \big| > \eps \big\}, \qquad i \in \{ a, d \}, \]
\[ T_0^1 = \inf \{ t \geq 0 \colon N_{1,t}^K = 0 \}, \]
and
\[ T_{\eps}^1= \inf \{ t \geq 0 \colon N_{1,t}^K \geq \eps \}. \]
These stopping times depend on $K$, but we omit the $K$-dependence from the notation for readability. %Our first goal is to show that for all sufficiently large $M>0$, we have

We couple $(N_{2a,t}^K,N_{2d,t}^K)$ with two bi-type birth-and-death processes, $(Y_{2a,t}^{\eps,\leq},Y_{2d,t}^{\eps,\leq})$ and $(Y_{2a,t}^{\eps,\geq},Y_{2d,t}^{\eps,\geq})$, such that
\[ Y^{\eps,\leq}_{2\upsilon,t} \leq N_{2\upsilon,t}^K \leq Y^{\eps,\geq}_{2\upsilon,t}, \qquad a.s. \qquad \forall \upsilon \in \{ a, d \},~\forall 0 \leq t \leq T_{\eps}^1. \numberthis\label{lastmutantcoupling} \] 
In order to satisfy \eqref{lastmutantcoupling}, these processes can be defined with the following rates:
\[ \begin{aligned}
(Y_{2a,t}^{\eps,\leq},Y_{2d,t}^{\eps,\leq}) \colon \quad &  \Big(\frac{i}{K},\frac{j}{K}\Big) \to \Big(\frac{i+1}{K},\frac{j}{K}\Big)& \text{at rate} & \quad i \lambda_2, \\
& \Big(\frac{i}{K},\frac{j}{K}\Big) \to \Big(\frac{i-1}{K},\frac{j}{K}\Big) & \text{at rate}  &\quad i(\mu+\alpha((1-p)\smfrac{i}{K}+\eps)), \\
&  \Big(\frac{i}{K},\frac{j}{K}\Big) \to \Big(\frac{i-1}{K},\frac{j+1}{K}\Big) & \text{at rate}& \quad \smfrac{p\alpha i^2}{K}, \\
&  \Big(\frac{i}{K},\frac{j}{K}\Big) \to \Big(\frac{i+1}{K},\frac{j-1}{K}\Big)  & \text{at rate} &\quad j\sigma, \\
&  \Big(\frac{i}{K},\frac{j}{K}\Big) \to \Big(\frac{i}{K},\frac{j-1}{K}\Big) & \text{at rate} &\quad j\kappa\mu. 
\end{aligned}
\]
and
\[ \begin{aligned}
(Y_{2a,t}^{\eps,\geq},Y_{2d,t}^{\eps,\geq}) \colon \quad &  \Big(\frac{i}{K},\frac{j}{K}\Big) \to \Big(\frac{i+1}{K},\frac{j}{K}\Big)& \text{at rate} & \quad i \lambda_2, \\
& \Big(\frac{i}{K},\frac{j}{K}\Big) \to \Big(\frac{i-1}{K},\frac{j}{K}\Big) & \text{at rate}  &\quad i(\mu+\alpha(1-p)\smfrac{i}{K}-p\alpha \eps), \\
&  \Big(\frac{i}{K},\frac{j}{K}\Big) \to \Big(\frac{i-1}{K},\frac{j+1}{K}\Big) & \text{at rate}& \quad \smfrac{p\alpha(i^2+\eps)}{K}, \\
&  \Big(\frac{i}{K},\frac{j}{K}\Big) \to \Big(\frac{i+1}{K},\frac{j-1}{K}\Big)  & \text{at rate} &\quad j\sigma, \\
&  \Big(\frac{i}{K},\frac{j}{K}\Big) \to \Big(\frac{i}{K},\frac{j-1}{K}\Big) & \text{at rate} &\quad j\kappa\mu. 
\end{aligned}
\]
The idea of this coupling is similar to the one in the proof of Proposition~\ref{prop-firstphase}: in order to decrease (increase) the process, one needs higher (lower) total competition event rate and rate of death by competition for the actives and lower (higher) active$\to$dormant switching rate. 

We will show that the processes $(Y_{2a,t}^{\eps,\leq},Y_{2d,t}^{\eps,\leq})$ and $(Y_{2a,t}^{\eps,\geq},Y_{2d,t}^{\eps,\geq})$ will stay close to $(\bar n_{2a},\bar n_{2d})$ for at least an exponential (in $K$) time with a probability close to 1 for large $K$. To do so, we will study the stopping times
\[ R_{\eta,\upsilon}^{\diamond} = \inf \big\{ t \geq 0 \colon N_{2\upsilon,t}^{\eps,\diamond} \notin [x_\upsilon-\eta,x_\upsilon+\eta] \big\} \]
for $\eta >0$, $\upsilon \in \{ a, d \}$ and $\diamond \in \{ \leq, \geq \}$. Let us first study the process $(Y_{2a,t}^{\eps,\leq},Y_{2d,t}^{\eps,\leq})$. According to \cite[Theorem 2.1, p.~456]{EK}, the dynamics of this process is close to the dynamics of the unique solution to
\[
\begin{aligned}
\dot n_{2a}&=n_{2a}(\lambda_2-\mu-\alpha \eps-\alpha n_{2a})+\sigma n_{2d}, \\
\dot n_{2d}&= p\alpha n_{2a}^2-(\kappa\mu+\sigma)n_{2d}.
\end{aligned}
\]
Similar to point \eqref{secondpoint} in Section~\ref{sec-invasionconditions}, we have that for all sufficiently small $\eps>0$, this system has a unique positive equilibrium, which we denote by $(\bar n_{2a}^{\eps,\leq},\bar n_{2d}^{\eps,\leq})$. Here, analogously to $\bar n_{2a}$, we have $\bar n_{2a}^{\eps,\leq}=\smfrac{(\lambda_2-\mu-\alpha\eps)(\kappa\mu+\sigma)}{\alpha(\kappa\mu+(1-p)\sigma)}$, whereas $\bar n_{2d}^{\eps,\leq}$ depends on $\eps$ in a more involved way, but it tends to $\bar n_{2d}$ as $\eps \downarrow 0$.
  %  \[ (\bar n_{2a}^{\eps,\leq},\bar n_{2d}^{\eps,\leq}): = \Big( \frac{(\lambda_2-\mu-\eps)(\kappa\mu+\sigma)}{\alpha(\kappa\mu+(1-p)\sigma)},  \frac{(\lambda_2-\mu-\eps)^2 p (\kappa\mu+\sigma)}{\alpha (\kappa\mu+(1-p)\sigma)^2} \Big). \]
For $\eps>0$ small enough, the equilibrium $(0,0)$ is unstable and $(\bar n_{2a}^{\eps,\leq},\bar n_{2d}^{\eps,\leq})$ is asymptotically stable, further, we can verify convergence of the solution to $(\bar n_{2a}^{\eps,\leq},\bar n_{2d}^{\eps,\leq})$ for any coordinatewise nonnegative initial condition but $(0,0)$, as $t \to \infty$, using similar arguments as in the proof of Lemma~\ref{lemma-2dODE}. Thus, we can find constants $c_0$ and $\eps'_0$ such that for any $\eps \in (0,\eps_0)$,
\[ \forall i \in \{ a,d \} \colon \big| \bar x_i^{\eps,\leq}-\bar x_i \big| \leq (c_0-1)\eps \text{ and } 0 \notin [\bar x_i-c_0 \eps, \bar x_i + c_0\eps] . \] %\color{red} I think $\mathcal A=2$ in the first phase is OK, but here we probably need a possibly larger constant $c_0$ (because of constant factors of $\alpha$ and possibly constants for equivalence of norms). \color{black} 

Now, similarly to the proof of Lemma~\ref{lemma-residentsstay}, we can use results by Freidlin--Wentzell about exit of jump processes from a domain \cite[Section 5]{FW84} in order to construct a family (over $K$) of Markov processes $(\widetilde Y_{2a,t},\widetilde Y_{2d,t})_{t \geq 0}$ whose transition rates are positive, bounded, Lipschitz continuous and uniformly bounded away from 0 such that for
\[ \widetilde R^\leq_{\eps,i} = \inf \big\{ t \geq 0 \colon \big| \widetilde Y_{2i,t}^K-\bar x_i \big| > \eps \big\}, \qquad i \in \{ a, d \}, \]
there exists $V>0$ such that for all $i \in \{ a, d \}$ we have
\[ \P\big( R^{\leq}_{c_0\eps,a} > \e^{KV},R^{\leq}_{c_0\eps,d} > \e^{KV}  \big) = \P\big( \widetilde R^{\leq}_{c_0\eps,a} > \e^{KV},\widetilde R^{\leq}_{c_0\eps,d} >  \e^{KV} \big) \underset{K \to \infty}{\longrightarrow} 1. \numberthis\label{firstFWextinction} \]
Similarly, we obtain
\[ \P\big( R^{\geq}_{c_0\eps,a} > \e^{KV},R^{\geq}_{c_0\eps,d} > \e^{KV}  \big) \underset{K \to \infty}{\longrightarrow} 1, \numberthis\label{secondFWextinction} \]
where without loss of generality we can assume that the constant $V$ in \eqref{secondFWextinction} is the same as the one in \eqref{firstFWextinction}. 
Now note that $R_{c_0\eps,i} \geq R_{c_0\eps,i}^\leq \wedge R_{c_0\eps,i}^\geq$ on the event $\{ R_{c_0\eps,i} \leq T_\eps^1 \}$. This together with \eqref{firstFWextinction} and \eqref{secondFWextinction} implies that 
\[ \lim_{K\to\infty} \P\big(R_{c_0\eps,i} \leq \e^{KV} \wedge T_\eps^1\big) =0\] holds for all $i \in \{ a, d \}$,
hence
\[ \lim_{K \to \infty} \P\big( R_{c_0\eps,a} \wedge R_{c_0\eps,d} \leq \e^{KV} \wedge T^1_\eps)=0. \numberthis\label{mutantsstayinequilibrium} \]
%\color{red} I found some arguments of the speciation paper \cite{C+16} superfluous, with leaving an $\bar n_{2a}/2$-neighbourhood of $\bar n_{2a}$ and the same for $\bar n_{2d}$, hopefully it's enough what I write but we have to check it. \color{black} 
Now, we can find two branching processes $Z_1^{\eps,\leq}=(Z_{1,t}^{\eps,\leq})_{t \geq 0}$ and $Z_1^{\eps,\geq}=(Z_{1,t}^{\eps,\geq})_{t \geq 0}$ such that 
\[ Z_{1,t}^{\eps,\leq} \leq K N_{1,t}^K \leq Z_{1,t}^{\eps,\geq} \numberthis\label{subcriticalcoupling} \]
almost surely on the time interval
\[ I_\eps^K = \big[ 0, R_{c_0\eps,a} \wedge R_{c_0\eps,d} \wedge  T^1_\eps \big]. \]
Indeed, in order to satisfy \eqref{subcriticalcoupling}, the processes $Z_{1}^{\eps,\geq}$ and $Z_{1}^{\eps,\leq}$ can be chosen with the following rates and initial conditions:
\[ \begin{aligned}
Z_{1}^{\eps,\leq} \colon & i \to i+1 & \text{ at rate } & i\lambda_1, 
& i \to i-1 & \text{ at rate } i \Big(\mu + \alpha \big(\bar n_{2a} + (c_0+1)\eps\big)\Big),
\end{aligned}
\]
started from $\lfloor \smfrac{ \eta \eps K}{2} \rfloor$, and
\[ \begin{aligned}
Z_{1}^{\eps,\geq} \colon & i \to i+1 & \text{ at rate } & i\lambda_1,
& i \to i-1 & \text{ at rate } i (\mu + \alpha (\bar n_{2a} - c_0\eps)),
\end{aligned}
\]
started from  $\lfloor \smfrac{\eps K}{2} \rfloor +1$.

For all $\eps>0$ sufficiently small, both of these branching processes are subcritical according to point \eqref{lastpoint} in Section~\ref{sec-invasionconditions}. The growth rates of these three processes are $\lambda_1-\mu-\alpha \bar n_{2a} \pm O(\eps)$. From this, analogously to \cite[Section 3.3]{C+19}, we deduce that the extinction time of these processes started from $[\lfloor\smfrac{ \eta K \eps}{2} \rfloor, \lfloor \smfrac{\eps K}{2} \rfloor +1]$ is of order $(\mu-\lambda_1+\alpha \bar n_{2a}\pm O(\eps))\log K$. This in turn follows from the general assertion that for a branching process $\mathcal N=(\mathcal N(t))_{t \geq 0}$ with birth rate $\Bcal>0$ and death rate $\Dcal>0$ that is subcritical (i.e., $\Bcal<\Dcal$), given that $\mathcal N(0) \in [\lfloor\smfrac{ \eta K \eps}{2} \rfloor, \lfloor \smfrac{\eps K}{2} \rfloor +1]$, defining
\[ \mathcal S_\eps^{\mathcal N} = \inf \{ t \geq 0 \colon \mathcal N(t) \geq \lfloor \eps K \rfloor \} , \qquad \eps > 0,  \]
and
\[ \mathcal S_\eps^{\mathcal N} = \inf \{ t \geq 0 \colon \mathcal N(t) =0 \} , \]
the following hold according to \cite[p.~202]{AN72}:
\[\forall \widetilde C<(\Dcal-\Bcal)^{-1}, \qquad \lim_{K \to \infty} \P(\mathcal S_0^{\mathcal N} \leq \widetilde C \log K)=0\]
and
\[ \forall\widetilde C>(\Dcal-\Bcal)^{-1}, \qquad \lim_{K \to \infty} \P(\mathcal S_0^{\mathcal N} \leq \widetilde C \log K)=1. \]
Further, if $\mathcal N(0)=\lfloor \smfrac{\eta K \eps}{2} \rfloor$, then for all sufficiently small $\eps>0$,
\[ \lim_{K \to \infty} \P\Big(\mathcal S_0^{\Ncal} > K \wedge \mathcal S_{\lfloor \eps K \rfloor }^{\mathcal N} \Big) = 0. \numberthis\label{diebeforeyougrowup} \]
Now for $C \geq 0$ we can estimate as follows
\[ 
\begin{aligned}
\P(T_0^1 < \widetilde C \log K) - \P\big(\mathcal S_0^{Z_1^{\eps,\leq}}  < \widetilde C \log K \big)
\leq & \P\big( T_0^1 > T^1_\eps \wedge K \big) + \P\big( T^1_\eps \wedge K > R_{c_0\eps,a} \wedge R_{c_0\eps,d} \big) \\
\leq & \P\big(\mathcal S_0^{Z_1^{\eps,\geq}} > \mathcal S_{\lfloor \eps K \rfloor}^{Z_1^{\eps,\geq}} \wedge K \big) +\P\big( T^1_\eps \wedge K > R_{c_0\eps,a} \wedge R_{c_0\eps,d} \big).
\end{aligned} \numberthis\label{strange} \]
Here, the first inequality can be verified as follows:
\[
\begin{aligned}
\P(T_0^1< \widetilde C \log K)&- \P \big( \mathcal S_0^{Z_1^{\eps,\leq}} < \widetilde C \log K \big)  =
 \P(T_0^1 < \widetilde C \log K \leq \mathcal S_0^{Z_1^{\eps,\leq}})
\\ &\leq    
\P(R_{c_0\eps,a}\wedge R_{c_0\eps,d} < T_0^1 < \widetilde C \log K, R_{c_0\eps,a}\wedge R_{c_0\eps,d}  < T_\eps^1) \\ & \qquad + \P(
 T_\eps^1 <T_0^1 < \widetilde C \log K, R_{c_0\eps,a}\wedge R_{c_0\eps,d}  > T_\eps^1)
\\ & \leq
\P(R_{c_0\eps,a}\wedge R_{c_0\eps,d}  < T_\eps^1 \wedge \widetilde C \log K) + \P(T_0^1 > T_\eps^1)
\\ & \leq
 \P(R_{c_0\eps,a}\wedge R_{c_0\eps,d}  < T_\eps^1 \wedge K) + \P(T_0^1 > T^1_\eps \wedge K).
\end{aligned}
\]
Given that $\eps>0$ is small enough, the second term in the last line of~\eqref{strange} tends to zero as $K \to \infty$ according to \eqref{mutantsstayinequilibrium} and so does the first one according to \eqref{diebeforeyougrowup}. We conclude that
\[ \limsup_{K \to \infty} \P\big( T_0^1 < \widetilde C \log K \big) \leq \lim_{K \to \infty} \P\big(\mathcal S_0^{Z_1^{\eps,\leq}} \leq \widetilde C \log K \big), \]
and similarly, we deduce
\[ \liminf_{K \to \infty} \P\big( T_0^1 < \widetilde C \log K \big) \geq \lim_{K \to \infty} \P\big(\mathcal S_0^{Z_1^{\eps,\geq}} \leq \widetilde C \log K \big), \]
%\color{red} small error in the speciation paper: I think that here we have to take limsup and liminf on the l.h.s.~because we have no clue whether the limit exists
\color{black} which implies the proposition.
\end{proof}

\subsection{Proof of Theorems~\ref{thm-invasion},\ref{thm-success}, and \ref{thm-failure}}\label{sec-proofremainder}
Putting together Propositions~\ref{prop-firstphase}, \ref{prop-secondphase}, and \ref{prop-thirdphase}, we now verify our main results. The structure of this part of our proof is similar to the one of~\cite[Section 3.4]{C+19}, the main difference lies in the behaviour of the corresponding dynamical systems. Our proof strongly relies on the coupling \eqref{upsilondef}. More precisely, we define a Bernoulli random variable $B$ as the indicator of nonextinction
\[ B:= \mathds 1 \{ \forall t>0 \colon \widehat Z_{2a}(t)+\widehat Z_{2d}(t)>0 \} \]
of the process $(\widehat Z_{2a}(t),\widehat Z_{2d}(t))_{t \geq 0}$ defined in point \ref{thirdpoint} of Section~\ref{sec-invasionconditions}, which is initially coupled with $(K N_{2a,t}^K,K N_{2d,t}^K)_{t \geq 0}$ according to \eqref{upsilondef}. 
Let $f$ be the function defined in Proposition~\ref{prop-firstphase}. Throughout the rest of the proof, we can assume that $\eps>0$ is so small that $f(\eps) <1$.

Our goal is to show that
\[ \liminf_{K \to \infty} \mathcal E(K,\eps) \geq q-o(\eps) \numberthis\label{extinctionlower}\]
holds for 
\[ \mathcal E(K,\eps):= \P \Big( \frac{T_0^2}{\log K} \leq f(\eps), T_0^2 < T_{S_\beta}, B=0 \Big)\]
and
\[ \liminf_{K \to \infty} \mathcal I(K,\eps) \geq 1-q-o(\eps) \numberthis\label{survivallower}\]
holds for
\[ \mathcal I(K,\eps):=\P \Big( \Big| \frac{T_{S_\beta} \wedge T_0^2}{\log K}-\Big( \frac{1}{\widetilde \lambda} + \frac{1}{\mu-\lambda_1+\alpha \bar n_{2a}} \Big)\Big| \leq f(\eps), T_{S_\beta} < T_0^2, B=1 \Big). \]
These together will imply Theorem~\ref{thm-invasion}, Theorem~\ref{thm-success}, and the equation \eqref{extinction} in Theorem~\ref{thm-failure}. The other assertion of Theorem~\ref{thm-failure}, equation \eqref{lastoftheorem}, follows already from \eqref{secondofprop}.

Let us start with the case of mutant extinction in the first phase of invasion and verify \eqref{extinctionlower}. Clearly, we have 
\[ \mathcal E(K,\eps)\geq \P \Big( \frac{T_0^2}{\log K} \leq f(\eps), T_0^2 < T_{S_\beta}, B=0, T_0^2 < T_{\eps}^2 \wedge R_{2\eps} \Big). \]
Now, considering our initial conditions, for all sufficiently small $\eps>0$ we have
$T_{\eps}^2 \wedge R_{2\eps} < T_{S_\beta}$, almost surely.
Hence,
\[ \mathcal E(K,\eps)\geq \P \Big( \frac{T_0^2}{\log K} \leq f(\eps), B=0, T_0^2 < T_{\eps}^2 \wedge R_{2\eps} \Big). \numberthis\label{andisand} \]
Moreover, analogously to the proof of Proposition~\ref{prop-firstphase} with $\xi=1$, we obtain
\[ \limsup_{K \to \infty} \P \Big( \{ B=0 \} \Delta \{ T_0^2 < T_{\eps}^2 \wedge R_{2\eps}  \}  \Big)+\P \Big( \{ B=0 \} \Delta \{ T_0^{(\eps,+),2} < \infty \}  \Big)=o_\eps(1), \numberthis\label{undefinedsymmdiff} \]
where $\Delta$ stands for symmetric difference. %and
%\[ \limsup_{K \to \infty} \P \big( \{ B=0 \} \Delta \{ T_0^{(\eps,+),2} < \infty \}  \big)=o_\eps(1). \]
%\color{red} Here I wrote $T_0^{(\eps,+),2} < \infty$. The homogamy paper says  $T_0^{(\eps,+),2} < T_\infty^{(\eps,+),2}$ several times, which I don't understand because the $(\eps,+)$-process a.s.~does not reach $\infty$ in finite time. \color{black}
Together with \eqref{andisand}, these imply
\begin{align*}
\liminf_{K \to \infty} \mathcal E(K,\eps)  &\geq \P \Big( \frac{%\widetilde 
\widetilde \lambda T_0^2}{\log K} \leq f(\eps), T_0^{(\eps,+),2} < \infty \Big) + o_\eps(1) \\ &\geq   \P \Big( \frac{\widetilde\lambda
T_0^{(\eps,+),2}}{\log K} \leq f(\eps), T_0^{(\eps,+),2} < \infty \Big) + o_\eps(1) \numberthis\label{secondline} \geq \P \Big( T_0^{(\eps,+),2} < \infty \Big) + o_\eps(1), 
\end{align*}
where in the first equality of \eqref{secondline} we used the coupling \eqref{upsilondef}. Thus, using~\eqref{qdef} and \eqref{qineq}, we conclude \eqref{extinction}. %\color{red} In the homogamy paper there are some $\widetilde\lambda$'s in front of $T_0^{(\eps,+),2}$, but they seem a bit dubious to me\color{black}

Let us continue with the case of mutant survival in the first phase of invasion and verify \eqref{survivallower}. %The convergence in probability in \eqref{invasion} is equivalent to the assertion that
%\[ \liminf_{K \to \infty} \mathcal I(K,\eps) \geq 1-q-o(\eps) \]
%holds for
%\[ \mathcal I(K,\eps):=\P \Big( \Big| \frac{T_{S_\beta} \wedge T_0^2}{\log K}-\Big( \frac{1}{\widetilde \lambda} + \frac{1}{\mu-\lambda_1+\alpha \bar n_{2a}} \Big)\Big| \leq f(\eps), T_{S_\beta} < T_0^2, B=1 \Big). \]
Arguing analogously to \eqref{undefinedsymmdiff} but for $\xi=1/2$, we obtain
\[ \limsup_{K \to \infty} \P \big( \{ B=1 \} \Delta \{ T_{\sqrt \eps}^2  < T_0^2 \wedge R_{2\eps}  \}  \big)=o_\eps(1). \] 
Thus,
\begin{equation}\label{beforesets}
\liminf_{K \to \infty} \mathcal I(K,\eps) = \liminf_{K \to \infty} \P \Big( \Big| \frac{T_{S_\beta}}{\log K} -\Big( \frac{1}{\widetilde \lambda} + \frac{1}{\mu-\lambda_1+\alpha \bar n_{2a}} \Big) \Big| \leq f(\eps), T_{S_\beta}<T_0^2, T^2_{\sqrt \eps} < T_0^2 \wedge R_{2\eps} \Big) + o_\eps(1). 
\end{equation} 
For $\eps>0,\beta>0$, we introduce the sets
\[ \begin{aligned}
\Bcal^1_\eps &:= [\pi_{2a}-\delta,\pi_{2a}+\delta] \times [\eps/C,\sqrt \eps] \times [\bar n_1-2\eps,\bar n_1+2\eps], \\
\Bcal^2_\beta &:=[0,\beta/2] \times [\bar n_{2a}-(\beta/2),\bar n_{2a}+(\beta/2)] \times [\bar n_{2d}-(\beta/2),\bar n_{2d}+(\beta/2)]
\end{aligned}
\]
and the stopping times
\[
\begin{aligned}
T'_\eps:=& \inf \Big\{ t \geq 0 \colon \Big( \frac{N_{2a,t}^K}{N_{2a,t}^K+N_{2d,t}^K}, N_{2,t}^K, N_{1,t}^K \Big)\in \Bcal^1_\eps \Big\}, \\
T''_\beta:=&\inf \Big\{ t \geq T'_\eps \colon \mathbf N^K_t \in \Bcal^2_\beta \Big\}.
\end{aligned}
\]
Informally speaking, our goal is to show that with high probability the process has to pass through $\Bcal^1_\eps$ and $\Bcal^2_\beta$ in order to reach $S_\beta$. Then, thanks to the Markov property, we can estimate $T_{S_\beta}$ by estimating $T'_\eps$, $T''_\beta-T'_\eps$ and $T_{S_\beta}-T''_\beta$. \eqref{beforesets} implies that 
\begin{align*}
    \liminf_{K \to \infty}&\mathcal I(K,\eps) 
    %\\ \geq &  \P \Big( \Big| \frac{T_{S_\beta}}{\log K} -\Big( \frac{1}{\widetilde \lambda} + \frac{1}{\mu-\lambda_1+\alpha \bar n_{2a}} \Big) \Big| \leq f(\eps), T^2_{\sqrt \eps} < T_0^2 \wedge R_{2\eps}, T''_{\beta}<T_{S_\beta}, T_{S_\beta}<T_0^2 \Big) + o_\eps(1) \\
    \geq  \P \Big( \Big| \frac{T'_\eps}{\log K} -\frac{1}{\widetilde \lambda} \Big| \leq \frac{f(\eps)}{3}, \Big| \frac{T''_\beta-T'_\eps}{\log K} \Big| \leq \frac{f(\eps)}{3}, \Big| \frac{T_{S_{\beta}}-T''_{\beta}}{\log K}- \frac{1}{\mu-\lambda_1+\alpha \bar n_{2a}}\Big| \leq \frac{f(\eps)}{3},  \\
    & \qquad   T^2_{\sqrt \eps} < T_0^2 \wedge R_{2\eps} , T''_{\beta}<T_{S_\beta}, T_{S_\beta}<T_0^2 \Big) + o_\eps(1) ,
\end{align*}
see \cite[display before (3.60)]{C+19} for further details in a similar setting. Note that for $\eps>0$ sufficiently small, $R_{2\eps} \leq T_{S_\beta}$ almost surely, further, if $T'_\eps<\infty$, then $T'_\eps < T''_\beta$. Hence, the strong Markov property applied at times $T'_\eps$ and $T''_\beta$ implies 
\[ \begin{aligned}
    \liminf_{K \to \infty} \mathcal I(K,\eps)&\geq \liminf_{K \to \infty} \Big[ \P \Big( \Big| \frac{T'_\eps}{\log K} -\frac{1}{\widetilde \lambda} \Big| \leq \frac{f(\eps)}{3}, T'_\eps<T_0^2, T^2_{\sqrt\eps} < T_0^2 \wedge R_{2\eps} \Big) \\
    & \qquad \times \inf_{\begin{smallmatrix}\mathbf n=(n_1,n_{2a},n_{2d}) \colon \big(\frac{n_{2a}}{n_{2a}+n_{2d}},n_{2a}+n_{2d},n_1\big) \in \Bcal^1_\eps\end{smallmatrix}} \P \Big(  \Big| \frac{T''_\beta-T'_\eps}{\log K} \Big| \leq \frac{f(\eps)}{3}, T''_{\beta} < T_0^2 \Big| \mathbf N^K_0=\mathbf n \Big) \\
    & \qquad \times \inf_{\mathbf n \in \Bcal^2_\beta} \P \Big(\Big| \frac{T_{S_{\beta}}-T''_{\beta}}{\log K}- \frac{1}{\mu-\lambda_1+\alpha \bar n_{2a}}\Big| \leq \frac{f(\eps)}{3}, T_{S_\beta} < T_0^2  \Big| \mathbf N_0^K = \mathbf n\Big) \Big]+o_{\eps}(1). \end{aligned} \numberthis\label{productform} \]
It remains to show that the right-hand side is close to $1-q$ as $K \to \infty$ and $\eps$ is small. We first consider the first term and verify that 
\[ \liminf_{K \to \infty} \P \Big( \Big| \frac{T'_\eps}{\log K} -\frac{1}{\widetilde \lambda} \Big| \leq \frac{f(\eps)}{3}, T'_\eps<T_0^2, T^2_{\sqrt\eps} < T_0^2 \wedge R_{2\eps} \Big) \geq 1-q+o_{\eps}(1). \numberthis\label{firstterm} \]
This can be done analogously to \cite[Proof of (3.61)]{C+19}.

Next, we handle the second term on the right-hand side of \eqref{productform}.
For $\mathbf m=(m_1,m_{2a},m_{2d}) \in [0,\infty)^3$, let $\mathbf n^{(\mathbf m)}$ denote the unique solution of the dynamical system \eqref{3dimlotkavolterra} with initial condition $\mathbf m$. Thanks to the continuity of flows of this dynamical system with respect to the initial condition and thanks to the convergence provided by Lemma~\ref{lemma-3dODE}, we deduce that there exist $\eps_0$ and $\delta_0>0$ such that for all $\eps \in (0,\eps_0)$ and $\delta \in (0,\delta_0)$, there exists $t_{\beta,\delta,\eps}>0$ such that for all $t>t_{\beta,\delta,\eps}$,
\[ \big\vert \mathbf n^{(\mathbf n^0)}(t)-(0,\bar n_{2a},\bar n_{2d}) \big\vert \leq \frac{\beta}{4} \]
holds for any initial condition $n^0=(n_1^0,n_{2a}^0,n_{2d}^0)$ such that $(n_{2a}^0/(n_{2a}^0+n_{2d}^0),n_{2a}^0+n_{2d}^0,n_1^0) \in \mathcal B^1_\eps$. Indeed, because of Lemma~\ref{lemma-goodstart}, $n^0$ satisfies \eqref{proportioncond} in case $n_{2a}^0/(n_{2a}^0+n_{2d}^0)$ is equal to $\pi_{2a}$, and for all sufficiently small $\eps>0$, the same follows by continuity for all $n^0=(n_1^0,n_{2a}^0,n_{2d}^0)$ such that  $(n_{2a}^0/(n_{2a}^0+n_{2d}^0),n_{2a}^0+n_{2d}^0,n_1^0) \in \mathcal B^1_\eps$.

Now, using Lemma~\ref{lemma-EKconvergence}, we conclude that for all $\eps<\eps_0$,
\[ \lim_{K \to \infty} \P\Big( T''_\beta-T'_\eps \leq t_{\beta,\delta,\eps} \Big|\Big( \frac{N_{2a,0}^K}{N_{2a,0}^K+N_{2d,0}^K}, N_{2,0}^K, N_{1,0}^K \Big) \in \mathcal B^1_\eps \Big) =1-o_\eps(1). \]
Thus, the second term on the right-hand side of \eqref{productform} is close to 1 when $K$ tends to $\infty$ and $\eps>0$ is small.

Lastly, we investigate the third term on the right-hand side of \eqref{productform}. By Proposition~\ref{prop-thirdphase}, there exists $\beta_0>0$ (denoted as $\eps_0$ in Proposition~\ref{prop-thirdphase}) such that for all $\beta<\beta_0$, for $\eps>0$ sufficiently small,
\[ \lim_{K \to \infty} \P \Big(\Big| \frac{T_{S_{\beta}}-T''_{\beta}}{\log K}- \frac{1}{\mu-\lambda_1+\alpha \bar n_{2a}}\Big| \leq \frac{f(\eps)}{3} \Big| \mathbf N_0^K \in \mathcal B^2_\beta \Big) = 1-o_\eps(1). \]
Further $\beta_0$ can be chosen as large as $\min \{ \bar n_{2a}, \bar n_{2d} \}$. 
Combining \eqref{firstterm} with the convergence of the second and the third term on the right-hand side of \eqref{productform} to 1, we obtain \eqref{survivallower},
%\[ \liminf_{K \to \infty} \mathcal I(K,\eps) \geq 1-q-o_\eps(1), \]
which implies \eqref{invasion}. 

\section*{Acknowledgements}
The authors thank A.~Kraut, N.~Kurt and J.~T.~Lennon for interesting discussions and comments. JB was supported by DFG Priority Programme 1590 ``Probabilistic Structures in Evolution’’, project 1105/5-1. AT was supported by DFG Priority Programme 1590 ``Probabilistic Structures in Evolution’’.

\end{document}